\documentclass[sn-mathphys]{JORSC-jnl}

\usepackage{graphicx}%
\usepackage{multirow}%
\usepackage{amsmath,amssymb,amsfonts}%
\usepackage{amsthm}%
\usepackage{mathrsfs}%
\usepackage[title]{appendix}%
\usepackage{xcolor}%
\usepackage{textcomp}%
\usepackage{manyfoot}%
\usepackage{booktabs}%
\usepackage{algorithm}%
\usepackage{algorithmicx}%
\usepackage{algpseudocode}%
\usepackage{listings}%
\usepackage{hyperref}
\usepackage{caption}
\usepackage{placeins}

\newtheoremstyle{thmstyleone} %
{3pt} 
{3pt} 
{\itshape} 
{} 
{\bfseries} 
{.} 
{5mm} 
{} 

\newtheoremstyle{thmstyletwo}
{3pt}
{3pt}
{\normalfont} 
{}
{\bfseries}
{.}
{5mm}
{}

\newtheoremstyle{thmstylethree}
{3pt}
{3pt}
{\normalfont} 
{}
{\bfseries}
{.}
{5mm}
{}

\newtheoremstyle{thmstylefour}
{3pt}
{3pt}
{\normalfont} 
{}
{\itshape} 
{.}
{5mm}
{}

\theoremstyle{thmstyleone}%
\newtheorem{theorem}{Theorem}
\newtheorem{proposition}[theorem]{Proposition}%
\newtheorem{lemma}[theorem]{Lemma}%
\newtheorem{corollary}[theorem]{Corollary}%

\theoremstyle{thmstyletwo}%
\newtheorem{example}{Example}%

\theoremstyle{thmstylethree}%
\newtheorem{definition}{Definition}%

\theoremstyle{thmstylefour}%
\newtheorem{assumption}{Assumption}%
\usepackage{indentfirst}
\usepackage{epstopdf}

\begin{document}
\renewcommand{\qedsymbol}{}
\renewcommand{\equalcont}[1]{\footnotetext{#1}}

\title[Local Convergence Analysis of ADMM for NCO]{Local Convergence Analysis of ADMM for Nonconvex Composite Optimization}


\author[1]{Xi-yuan Xie}\email{xiexy65@mail2.sysu.edu.cn}

\author[1]{Li-hua Yang}\email{mcsylh@mail.sysu.edu.cn}

\author*[2]{Qia Li}\email{liqia@mail.sysu.edu.cn}

\affil[1]{\orgdiv{School of Mathematics}, \orgname{Sun Yat-sen University}, \orgaddress{\street{135 Xingang West Road}, \city{Guangzhou}, \postcode{510275}, \state{Guangdong}, \country{China}}}

\affil[2]{\orgdiv{ School of Computer Science and Engineering, Guangdong Province Key Laboratory of Computational Science}, \orgname{Sun Yat-sen University}, \orgaddress{\street{135 Xingang West Road}, \city{Guangzhou}, \postcode{510275}, \state{Guangdong}, \country{China}}}

\equalcont{The work of Qia Li is supported by the National Natural Science Foundation of China (No.12471098). The work of Li-hua Yang is supported by the Guangdong Basic and Applied Basic Research Foundation (No.2024A1515010988).}


\abstract{
In this paper, we study the local convergence of the standard ADMM scheme for a class of nonconvex composite optimization problems motivated by applications in signal processing and machine learning. The problems are constrained by a closed convex set, while their objective is the sum of a continuously differentiable, possibly nonconvex, smooth term and a polyhedral convex nonsmooth term composed with a linear mapping. Motivated by recent works of Rockafellar, we first provide an elementary proof of a local strong convexity property of the Moreau envelope of polyhedral convex functions on the orthogonal complement of an appropriate subspace. Building on this property, we establish the strong variational sufficiency of the reduced augmented Lagrangian under an appropriate second-order condition. We then derive a descent inequality for the ADMM iterates that is analogous to the classical descent inequality for convex ADMM. For a sufficiently large penalty parameter, and under suitable initialization and local trajectory conditions, we prove that the ADMM sequence converges to a stationary primal-dual point. When the constraint set is polyhedral convex, we further show that the weighted distance of the primal-dual sequence to the local solution set converges Q-linearly, while the primal sequence converges R-linearly. Finally, we present three illustrative examples together with an application-oriented verification for a class of possibly nonconvex quadratic programs, illustrating the role of the second-order condition, the local nature of the convergence theory, and its applicability.}




\maketitle

\section{Introduction}\label{sec-intro}
Nonconvex composite optimization has attracted great attention due to its applications to various modern image processing and machine learning models, while the alternating direction method of multipliers (ADMM) is one of the most widely used methods for solving this type of problem, see, for example, \cite{lanza2016convex,chan2016plug,shen2019iterative,wang2021limited,gräf2022image,mancino2023decentralized,kumar2024efficient,bui2024stochastic,barber2024convergence}. In this paper, we consider a class of nonconvex composite optimization problems as follows: \begin{align}\label{prob:com}
\min_{x\in{\mathcal{C}}}~~ F(x):=f(x) + g(Ax)
\end{align}
where the mapping $A: \mathbb{R}^n \to \mathbb{R}^m$ is linear, 
$\mathcal{C}\subseteq \mathbb{R}^n$ is a nonempty closed convex set, \mbox{$f: \mathbb{R}^n \to \mathbb{R}$} is continuously differentiable, and \mbox{$g: \mathbb{R}^m \to \mathbb{R} \cup \{\infty\}$} is polyhedral convex. We further assume that \[ \mathrm{dom}~g \cap A(\mathrm{ri}(\mathcal{C}))\neq\emptyset.\]

Problem~\eqref{prob:com} has many important applications in machine learning and signal processing. Below, we present two representative application models.

~

\noindent\textbf{Quadratic programming.} 
Consider the following quadratic optimization problem
\begin{equation}
    \begin{aligned}\label{QP}
    \min_{x\in \mathbb{R}^n}~&\frac{1}{2}x^\top Hx+b^\top x,\\
   \text{ s.t. } &Ex=e,\; Gx\le h
\end{aligned}
\end{equation}
where $H\in\mathbb{R}^{n\times n}$ is symmetric, $E\in\mathbb{R}^{p\times n}$, $G\in\mathbb{R}^{q\times n}$, $b\in\mathbb{R}^n$,  $e\in\mathbb{R}^p$ and  \mbox{$h\in \mathbb{R}^q$}. Such models arise in a variety of applications, including compressed sensing \cite{chen2013sparse}, portfolio selection \cite{fabozzi2008portfolio}, and resource allocation \cite{ibaraki1988resource}. Problem \eqref{QP} can be written as an instance of Problem~\eqref{prob:com} by setting $A=I_n$, $\mathcal{C}=\mathbb{R}^n$, $f(x)=\frac{1}{2}x^\top Hx+b^\top x$, and $g(y)=\delta_{\mathcal{Q}}(y)$ with 
\[\mathcal{Q}=\{x\in\mathbb{R}^n\mid Ex=e,\; Gx\le h\}.\]

~

\noindent\textbf{Dictionary Learning Model.}
Let $X=[x_1,\ldots,x_m]\in\mathbb{R}^{n\times m}$ be a data matrix. A standard dictionary learning model \cite{mairal2010online} is
\begin{align}\label{exp-prob5}
    \min_{D\in \mathcal{D},\,W\in\mathbb{R}^{p\times m}}~\frac{1}{2}\|X-DW\|_F^2+\mu\|W\|_1,
\end{align}
where $\mu>0$ is a regularization parameter, $\|W\|_1:=\sum_{i,j}\vert W_{ij}\vert$, and
\[
\mathcal{D}:=\{D\in\mathbb{R}^{n\times p}\mid \|d_j\|_2\leq 1,~j=1,\ldots,p\}
\] is the convex constraint set on the dictionary columns. Problem \eqref{exp-prob5} is an instance of Problem~\eqref{prob:com} with
$f(D,W)=\frac{1}{2}\|X-DW\|_F^2$, linear mapping $A:\mathbb{R}^{n\times p}\times\mathbb{R}^{p\times m}\to \mathbb{R}^{p\times m}$ satisfying $A(D,W)=W$, $g(W)=\mu\|W\|_1$, and $\mathcal{C}=\mathcal{D}\times\mathbb{R}^{p\times m}$.

ADMM can be readily applied to Problem~\eqref{prob:com}. We first reformulate \eqref{prob:com} into the following equality-constrained problem:
\begin{equation}\label{prob:linear-eq}
\begin{aligned}
\min&~~ f(x)+\delta_{\mathcal{C}}(x) + g(y)\\
s.t.&~~Ax=y,
\end{aligned}  
\end{equation}
where $\delta_{\mathcal{ C}}$ is the indicator function of the closed convex set $\mathcal{C}$. The augmented Lagrangian function of Problem \eqref{prob:linear-eq} is then defined by  
\begin{align*}
L_{\beta}(x,y,\lambda):=&f(x)+\delta_{\mathcal{C}}(x)+g(y)\\&+\langle\lambda,Ax-y\rangle+\frac{\beta}{2}\|Ax-y\|^2,
\end{align*}
where $\lambda\in\mathbb{R}^m$ is the Lagrange multiplier and $\beta > 0 $ is a penalty parameter.
Now, following the standard procedure of ADMM, we obtain an ADMM iterative scheme for solving \eqref{prob:linear-eq} as follows:
\begin{subequations}\label{alg:classical ADMM1}
\begin{align}
    x^{k+1}&\in\arg\min_{x\in \mathcal{C}} f(x)+\langle \lambda^k,Ax\rangle+\frac{\beta}{2}\|Ax-y^k\|^2, \label{alg:cl1-ADMM1}\\
    y^{k+1}&=\arg\min_y g(y)+\langle\lambda^k,-y\rangle+\frac{\beta}{2}\|Ax^{k+1}-y\|^2,\label{alg:cl1-ADMM2}\\
    \lambda^{k+1}&=\lambda^k+\beta(Ax^{k+1}-y^{k+1}).\label{alg:cl1-ADMM3}
\end{align}
\end{subequations}

In the case that $f$ is convex, the convergence of \eqref{alg:classical ADMM1} and other variants of ADMM is extensively studied and  well understood, see, for example, \cite{eckstein1992douglas,boyd2011distributed,chambolle2011first,he20121,glowinski2014alternating,chen2015inertial,eckstein2015understanding,yang2016linear,chen2016direct,hong2017linear,davis2017faster,lin2022alternating,sun2021efficient,han2022survey}. However, when $f$ is nonconvex, the convergence analysis of ADMM is generally much more challenging. Although there are numerous studies on the convergence of ADMM for nonconvex composite optimization \cite{li2015global,hong2016convergence,wang2018convergence,jiang2019structured,wang2019global,zhang2019fundamental,boct2020proximal,jia2021local}, we find they are not applicable to investigating the convergence of \eqref{alg:classical ADMM1} for solving \eqref{prob:com}, which is mainly due to the presence of the convex constraint set $\mathcal{C}$. In fact, for Problem~\eqref{prob:com} with $\mathcal{C}=\mathbb{R}^n$, some results in \cite{li2015global,wang2019global,jia2021local} are available for analyzing the convergence of a variant of \eqref{alg:classical ADMM1}, which simply changes the order of minimizing $x$ and $y$ as follows
\begin{subequations}\label{alg:classical ADMM2}
\begin{align}
 y^{k+1}&=\arg\min_y g(y)+\langle\lambda^k,-y\rangle+\frac{\beta}{2}\|Ax^{k}-y\|^2,\label{alg:cl2-ADMM1}\\
    x^{k+1}&\in\arg\min_{x\in \mathcal{C}} f(x)+\langle \lambda^k,Ax\rangle+\frac{\beta}{2}\|Ax-y^{k+1}\|^2, \label{alg:cl2-ADMM2}\\
    \lambda^{k+1}&=\lambda^k+\beta(Ax^{k+1}-y^{k+1}).\label{alg:cl2-ADMM3}
\end{align}
\end{subequations}
Specifically, these results indicate that if $A$ is a nonsingular square matrix, then any cluster point of the solution sequence generated by \eqref{alg:classical ADMM2} is a stationary point, provided that the penalty parameter $\beta$ is chosen sufficiently large. Moreover, in \cite{li2015global,wang2019global} global convergence of the entire sequence is established by assuming the KL property \cite{attouch2010proximal}, while in \cite{jia2021local} global convergence and local linear convergence are shown via an error bound condition. To the best of our knowledge, the existing results do not
directly cover the ADMM schemes~\eqref{alg:classical ADMM1} and \eqref{alg:classical ADMM2} under the present assumptions when \(\mathcal C\neq\mathbb R^n\).

Recently, in \cite{rockafellar2023convergence} Rockafellar conducts local convergence analysis of the augmented Lagrangian method (ALM) for nonconvex composite optimization which restricts $\mathcal{C}=\mathbb{R}^n$ and replaces $g\circ A$ of Problem~\eqref{prob:com} by the composition of a convex function and a nonlinear mapping. The principal idea therein is to assume that  the \textit{strong~variational~sufficiency~condition} \cite[Theorem~1]{rockafellar2023augmented} holds for the augmented Lagrangian function, which actually reduces to the strong second-order sufficient condition (SOSC) when applied to the classical nonlinear programming problems \cite[Theorem~5]{rockafellar2023augmented}. Subsequently, the authors in \cite{zhou2025some} extend the (strong) variational sufficiency condition to nonsmooth composite optimization problems on Riemannian manifolds. In \cite{van2026convergence}, the authors use a weaker version of the SOSC to analyze the local convergence of ALM. 

Note that the convergence issue of ADMM scheme \eqref{alg:classical ADMM1}  remains open, and motivated by the work \cite{rockafellar2023augmented,rockafellar2023convergence}, we adapt the notion of strong variational sufficiency to Problem~\eqref{prob:com} and prove the local convergence of \eqref{alg:classical ADMM1} under this condition. To the best of our knowledge, this is the first work to study the convergence of ADMM for nonconvex composite optimization through strong variational sufficiency. The main contributions are summarized as follows. First, we provide an elementary proof of a local strong convexity property of the Moreau envelope of polyhedral convex functions on the orthogonal complement of an appropriate subspace. More precisely, we establish a local Hessian
lower bound determined by the orthogonal projection onto this subspace. Building on this property, we show that the strong variational sufficiency holds for Problem~\eqref{prob:com} under a second-order condition requiring the Hessian of the smooth term to be positive definite on an appropriate subspace. Second, using the resulting local structure in which the reduced augmented Lagrangian is strongly convex in \(x\) and concave in \(\lambda\), we derive a descent inequality for the ADMM iterates that parallels the classical analysis of convex ADMM. For sufficiently large penalty parameters, we prove the local convergence of ADMM scheme~\eqref{alg:classical ADMM1} under suitable initialization and local trajectory conditions. When \(\mathcal C\) is polyhedral convex, we further establish Q-linear convergence of a weighted distance to the local primal-dual solution set and R-linear convergence of the primal sequence. Finally, we present three illustrative examples together with an application-oriented verification for a class of possibly nonconvex quadratic programs. These examples illustrate the role of Assumption~\ref{assp:ADMM}, the local nature of the convergence theory, and the applicability of our results.

The rest of the paper is organized as follows. Section \ref{sec:Pre} presents preliminaries that will be used in this paper. Section \ref{sec:condition} discusses the strong variational sufficiency of Problem~\eqref{prob:com}. Section \ref{sec:local} establishes the convergence of ADMM scheme \eqref{alg:classical ADMM1}. Section \ref{sec:example} presents three illustrative examples together with a verification for a class of practical quadratic programs.

\section{Notation and Preliminaries}\label{sec:Pre}
\textbf{Notation.}~The set of extended real numbers is denoted by $\overline{\mathbb{R}}:=\mathbb{R}\cup\{\pm\infty\}$.
Let $\|\cdot\|$ denote the Euclidean norm on $\mathbb{R}^n$.
For a function $\phi:\mathbb{R}^n\to\overline{\mathbb{R}}$, the classical directional derivative at $x$ along $d\in\mathbb{R}^n$ is denoted by $\phi'(x;d)$.
The set of global minimizers of $\phi$ is denoted by $\operatorname{argmin}\phi:=\{x\in\mathbb{R}^n\mid \phi(x)=\inf_z \phi(z)\}$.
The conjugate function of $\phi$ is denoted by $\phi^*$, and its effective domain by $\mathrm{dom}~\phi:=\{x\in\mathbb{R}^n\mid \phi(x)<\infty\}$.
In addition, $\partial\phi$, $\nabla\phi$ and $\nabla^2\phi$ denote the subdifferential, gradient and Hessian of $\phi$, respectively, wherever they exist.

For a nonempty closed convex set $\Omega\subset\mathbb{R}^n$, let $\Omega^o:=\{w\in\mathbb{R}^n\mid w^{\top}v\leq 0,\ \forall v\in\Omega\}$, let $\delta_\Omega$ be the indicator function of $\Omega$, let $\Pi_{\Omega}$ be the projection operator onto $\Omega$, and let $T_{\Omega}(x)$ and $N_{\Omega}(x)$ be the tangent cone and normal cone to $\Omega$ at $x\in\Omega$, respectively.
Moreover, let $\mathrm{aff}_{\Omega}$ denote the smallest affine space containing $\Omega$ and define $V_{\Omega}:=\mathrm{aff}_{\Omega}-x$ for any $x\in\Omega$.
We denote by $\mathrm{int}(\Omega)$ and $\mathrm{ri}(\Omega)$ the interior and relative interior of $\Omega$, and by $\mathrm{conv}(\Omega)$ the convex hull of $\Omega$ (the smallest convex set containing $\Omega$). For $x\in\mathbb{R}^n$, let $\mathrm{dist}(x,\Omega):=\inf_{y\in\Omega}\|x-y\|$.

For symmetric matrices $A,B\in\mathbb{R}^{n\times n}$, $A\succeq B$ means that $A-B$ is positive semidefinite.
For $x\in\mathbb{R}^n$ and $r>0$, let $\mathbb{B}(x,r):=\{y\in\mathbb{R}^n\mid \|y-x\|\leq r\}$.

\textbf{Preliminaries.}
For Problem~\eqref{prob:com}, it holds that for any $ x\in\mathbb{R}^n$, 
\begin{align}\label{subgrad-sum}
\partial(g(Ax)+\delta_{\mathcal{C}}(x))=\partial\delta_{\mathcal{C}}(x)+A^{\top}\partial g(Ax),
\end{align} from \cite[Theorem~23.8 and 23.9]{rockafellar1997convex}. 
Furthermore, suppose that $\bar{x}$ is a local minimizer of Problem~\eqref{prob:com}. Then, by~\eqref{subgrad-sum} and \cite[Theorem~3.72]{Beck2017First}, there exists a multiplier $\bar{\lambda}$ such that
\begin{align}\label{first order condition1}
0\in\partial\delta_{\mathcal{C}}(\bar{x})+\nabla f(\bar{x})+A^{\top}\bar{\lambda},~~\bar{\lambda}\in \partial g(A\bar{x}).
\end{align}
For a function $\varphi:\mathbb{R}^m\to\overline{\mathbb{R}}$, its Moreau envelope function is defined as 
\[M^{\mu}_{\varphi}(x):=\inf_y~\varphi(y)+\frac{1}{2\mu}\|y-x\|^2,\]
where the constant $\mu>0$, and its proximal mapping is defined by  
\[\mathrm{prox}_{\mu \varphi}(x):=\arg\min_y~\varphi(y)+\frac{1}{2\mu}\|y-x\|^2.\]
\begin{proposition}[{\cite[Proposition~6]{eckstein2015understanding}}]\label{Moreau-decom} Let $\mu>\nolinebreak 0$ and $\varphi:\mathbb{R}^m\to\overline{\mathbb{R}}$ be a proper closed convex function. Then for every $w\in\mathbb{R}^m$, there exists a unique pair $(u_w,\lambda_{w})$ satisfying \[w=u_{w}+ \mu\lambda_{w},~~u_{w}\in\mathrm{dom}~\varphi, ~~\lambda_{w}\in\partial \varphi(u_{w}),\]
    where $u_w=\mathrm{prox}_{\mu \varphi}(w)$ and $\mu\lambda_{w}=\mathrm{prox}_{(\mu \varphi)^*}(w)$. \end{proposition}
    By Proposition \ref{Moreau-decom} and \cite[Theorem~4.20]{Beck2017First}, 
for a proper closed convex function $\varphi$, it holds for any $\mu>0$  that 
\begin{equation}\label{eq:partial}
    \begin{aligned}
  x\in\partial \varphi^*(\lambda)  &\Longleftrightarrow~\lambda\in\partial \varphi(x)\\
  &\Longleftrightarrow~\lambda=\nabla M^{\mu}_{\varphi}(x+\mu\lambda)
  \\
  &\Longleftrightarrow x=\mathrm{prox}_{\mu \varphi} (x+\mu\lambda).
\end{aligned}
\end{equation}
A differentiable function $\varphi:\mathbb{R}^n\to\overline{\mathbb{R}}$ is convex if and only if 
\begin{align*}
    \langle \nabla \varphi(x)-\nabla \varphi(y), x-y\rangle\geq 0,~~\forall x,y\in\mathbb{R}^n.
\end{align*}
Moreover, according to Rademacher's Theorem \cite[Theorem~9.60]{rockafellar2009variational}, if $\nabla \varphi$ is locally Lipschitz continuous, $\nabla^2 \varphi$ exists almost everywhere. In this case, $\varphi$ is convex if and only if $\nabla^2 \varphi$ is positive semidefinite wherever it exists \cite[Section~3]{rockafellar2023augmented}. Given a nonempty  closed convex set $\Omega\subseteq \mathbb{R}^n$, it holds that for any $x\in\mathbb{R}^n$ and $y\in \Omega$,
\begin{align}\label{convex-proj}
    \langle y- \Pi_{\Omega}(x),x-\Pi_{\Omega}(x)\rangle\leq 0.
\end{align}

A set $\mathcal{P} \subseteq \mathbb{R}^n$ is called a  polyhedral convex set if it can be represented as the intersection of some finite collection of closed half-spaces. That is, there exist a matrix $B \in \mathbb{R}^{m \times n}$ and a vector $c \in \mathbb{R}^m$ such that $\mathcal{P}= \left\{ x \in \mathbb{R}^n \mid Bx \leq c \right\}$, where $Bx \leq c$ denotes the system of inequalities $B_i^\top x \leq c_i$ for $i = 1, 2, \dots, m$, where $B_i^\top$ is the $i$-th row of $B$. In particular, $\mathcal{P}$ is called a convex cone if $c=0$. Let $I(x):=\{i\mid B_i^{\top}x=c_i\}$, then tangent cone $T_\mathcal{P}(x)=\{v\in\mathbb{R}^n\mid B_i^{\top}v\leq0, i\in I(x)\}$ and  normal cone $N_{\mathcal{P}}(x)=\{\sum\lambda_i B^{\top}_i\mid \lambda_i\geq 0, i\in I(x)\}$, which are polyhedral convex cones \cite[Theorem~6.46]{rockafellar2009variational}.
A polyhedral convex function $\phi:\mathbb{R}^n\to\overline{\mathbb{R}}$ is a convex function whose epigraph is a polyhedral convex set. Equivalently, $\mathrm{dom} \phi$ is a polyhedral convex set and $\phi$ is piecewise affine on its domain \cite[Corollary~19.1.2]{rockafellar1997convex}. Namely, $\phi$ has the form  
\begin{align*}
    \phi(x)=\max_{1\leq i\leq m}\{a_i^{\top}x+b_i\}+\delta_{\mathcal{P}}(x)
\end{align*}
where $\mathcal{P}$ is a polyhedral convex set. 

With piecewise linearity, for $x\in\mathrm{dom} ~ \phi$ and $w\in\mathbb{R}^n$, the one-sided directional derivatives have the distinguishing feature as: $\forall{\rho_0}>0,~\exists\varepsilon>0$, such that
\begin{align}\label{eq:one-sided derivate} \phi'(x;w)=\frac{\phi(x+tw)-\phi(x)}{t},
\end{align}
where $\|w\|<{\rho_0},~0<t<\varepsilon$.
In addition, $\partial \phi(x)$ is a polyhedral set \cite[Theorem~23.10]{rockafellar1997convex}, and if $\phi'(x;w)<\infty$, there exists $\lambda \in \partial \phi(x)$ such that $\phi'(x;w)=\langle\lambda,w\rangle$.  Moreover, it holds that $\forall{\rho_0}>0,\exists\varepsilon>0$, such that if $\eta\in T_{\partial {\phi}(x)}(\lambda)$ with $\|\eta\|<{\rho_0}$, and $0<t<\varepsilon$, then
\begin{align}\label{eq:diff-tan}
\lambda+t\eta \in \partial {\phi}(x),
\end{align}

 The conjugate function $\phi^*$ of the polyhedral convex function $\phi$ is also a polyhedral convex function {\cite[Theorem~19.2]{rockafellar1997convex}}, and it holds that 
\begin{align}\label{polar}
    T_{\partial \phi(y)}(\lambda)=[N_{\partial \phi(y)}(\lambda)]^o= [T_{\partial \phi^*(\lambda)}(y)]^o.
\end{align}
Let \begin{align}\label{tangent-cone}
     T_{{\phi}}(x\mid \lambda):=\{w \mid {\phi}'(x;w)=\langle\lambda,w\rangle\},
     \\T_{{\phi}^*}(\lambda\mid x):=\{\eta \mid {\phi}^{*}{'}(\lambda;\eta)=\langle x,\eta \rangle\}.
 \end{align}
 Note that $\lambda\in\partial {\phi}(x)$ plays the role of a "gradient" restricted to the convex cone $T_{\phi}(x\mid\lambda)$. Clearly, it is possible that $T_{\phi}(x\mid\lambda)=\{0\}$. For example, if ${\phi}(x):=\vert x\vert$ with $x\in\mathbb{R}$, then $T_{\phi}(0\mid \frac{1}{2})=\{0\}$.
For a polyhedral convex function $\phi:\mathbb{R}^n\to\overline{\mathbb{R}}$, by $\lambda\in \partial {\phi}(x)\Longleftrightarrow x\in \partial {\phi}^*(\lambda)$ and \eqref{eq:one-sided derivate}, it holds that
\begin{equation}\label{eq:Tg}
    \begin{aligned}
    &T_{\phi}(x\mid \lambda)
    =\{w\mid \lambda\in\partial {\phi}(x + tw)\text{ for small }t > 0\}=T_{\partial {\phi}^*(\lambda)}(x),\\
&T_{{\phi}^*}(\lambda\mid x)
=\{\eta\mid x\in\partial {\phi}^*(\lambda + t\eta)\text{ for small }t > 0\}=T_{\partial {\phi}(x)}(\lambda).
\end{aligned}
\end{equation}
Through the relations in \eqref{eq:Tg}, the tangent cones $T_{\partial {\phi}^*(\lambda)}(x)$ and $T_{\partial {\phi}(x)}(\lambda)$ can be characterized by $T_{\phi}(x\mid \lambda)$ and $T_{{\phi}^*}(\lambda\mid x)$ respectively, a representation that can be more accessible.

A multifunction $\mathfrak{R}:\mathbb{R}^n\rightrightarrows\mathbb{R}^m$ is said to be \emph{piecewise linear} if its graph 
\[
\operatorname{Gr}(\mathfrak{R}) := \{(x,u)\in\mathbb{R}^n\times\mathbb{R}^m \mid u\in\mathfrak{R}(x)\}
\]
can be expressed as the union of finitely many polyhedral convex sets (as defined above) in the product space $\mathbb{R}^n\times\mathbb{R}^m$ \cite[Definition~2.2]{zheng2014metric}. Note that this notion coincides with \emph{piecewise polyhedral} multifunctions in \cite{rockafellar2009variational}.
In particular, if $\mathcal{C}\subseteq\mathbb{R}^n$ is a polyhedral convex set, its normal cone mapping $N_{\mathcal{C}}$ is piecewise linear; if $g:\mathbb{R}^n\to\overline{\mathbb{R}}$ is a polyhedral convex function, its subdifferential $\partial g$ is also piecewise linear \cite[Proposition~12.30]{rockafellar2009variational}. Moreover, it follows directly from the definition that the class of piecewise linear multifunctions is closed under the addition of a single‑valued affine mapping and under Cartesian products. The following error bound result plays a
fundamental role in our linear convergence rate analysis.
\begin{proposition}[{\cite[Theorem~3.3]{zheng2014metric}}]\label{prop-error}
     Let $\mathfrak{R}:\mathbb{R}^n\rightrightarrows \mathbb{R}^n $ be a piecewise linear multifunction. For any $\eta > 0$, there
exists $\kappa > 0$ such that
\begin{align*}
    \mathrm{dist}(u,\mathfrak{R}^{-1}(0))\leq \kappa \mathrm{dist}(0,\mathfrak{R}(u)),~~\forall~\|u\|\leq\eta.
\end{align*}
\end{proposition}
The speed of convergence is characterized by using the notions of Q-linear and R-linear rates.
A sequence $\{z^k\}\subseteq\mathbb{R}^n$ converging to $z^*\in\mathbb{R}^n$ is said to converge Q-linearly (quotient-linear) if there exists a constant $\rho\in(0,1)$ such that \[\limsup_{k\to\infty} \frac{\|z^{k+1}-z^*\|}{\|z^k-z^*\|}\leq \rho.\] It converges R-linearly (root-linear) if there exists a sequence $\{\varepsilon_k\}$ converging Q-linearly to zero such that $\|z^k-z^*\|\le \varepsilon_k$ for all sufficiently large $k$, or equivalently, if $\|z^k-z^*\|\le \alpha \rho^k$ for some $\alpha>0$ and $\rho\in(0,1)$. 
\section[Strong Variational Sufficiency for Problem (1)]{Strong Variational Sufficiency for Problem~\eqref{prob:com}}\label{sec:condition}
In this section, we extend the strong variational sufficiency defined in \cite[Theorem~1]{rockafellar2023augmented} to Problem~\eqref{prob:com}, and show that it holds under mild assumptions. To this end, we first introduce the reduced augmented Lagrangian function \mbox{$l_{\beta}(x,\lambda):\mathbb{R}^n\times\mathbb{R}^m\to\overline{\mathbb{R}}$} of Problem~\eqref{prob:com} as follows:
     \begin{equation}\label{eq:reduced-lag}
         \begin{aligned}
        l_{\beta}(x,\lambda):=&\min_{y}L_{\beta}(x,y,\lambda)      \\=&f(x)+M^{1/\beta}_g(Ax+\frac{\lambda}{\beta})
        +\delta_{\mathcal{C}}(x)-\frac{1}{2\beta}\|\lambda\|^2.
    \end{aligned}
     \end{equation}

The definition of strong variational sufficiency for Problem~\eqref{prob:com} is formally stated below.
\begin{definition}
Let $(\bar{x},\bar{\lambda})\in\mathbb{R}^n\times\mathbb{R}^m$ satisfy the first-order optimal condition \eqref{first order condition1}. We say the strong variational sufficiency for Problem~\eqref{prob:com} holds with respect to $(\bar{x},\bar{\lambda})$ if,  for every  sufficiently large $\beta> 0$, there exists a convex neighborhood $\mathcal{X}_{\beta}\times\Lambda_{\beta}$ of $(\bar{x}, \bar{\lambda})$ such that the reduced Lagrangian $l_{\beta}$ is strongly convex in $x$ when $\lambda\in\Lambda_{\beta}$  and concave in $\lambda$ when $x\in\mathcal{X}_{\beta}\cap\mathcal{C}$.    
\end{definition}

\begin{corollary}\label{coro-unique-x}
Assume that the strong variational sufficiency for Problem~\eqref{prob:com} holds with respect to $(\bar{x},\bar{\lambda})$ with $\beta>0$ and $\mathcal{X}_{\beta}\times \Lambda_{\beta}$.  If $(x^*,\lambda^*)$ is a saddle point of $l_{\beta}(x,\lambda)$ relative to $\mathcal{X}_{\beta}\times \Lambda_{\beta}$, then  $x^*=\bar{x}$.
\end{corollary}
\begin{proof}
    By the relationship \eqref{eq:partial}, the first-order optimal condition \eqref{first order condition1} can be equivalently reformulated as 
\begin{align}\label{first order condition2}
    0\in\partial_xl_{\beta}(\bar{x},\bar{\lambda}),~~0\in \partial_{\lambda}[-l_{\beta}](\bar{x},\bar{\lambda}).
\end{align}
Hence,  $(\bar{x}, \bar{\lambda})$ is a saddle point of $l_\beta(x,\lambda)$ relative to $\mathcal{X}_{\beta}\times\Lambda_{\beta}$. 
By the definition of saddle point, it holds that
\begin{align*}
l_\beta(\bar{x},\bar{\lambda})\leq l_{\beta}(x^*,\bar{\lambda})\leq l_{\beta}(x^*,\lambda^*) \leq l_\beta(\bar{x},\lambda^*)\leq l_\beta(\bar{x},\bar{\lambda}),
\end{align*}
which implies that $l_{\beta}(\bar{x},\bar{\lambda})=l_{\beta}(x^*,\bar{\lambda})$. Since $l_{\beta}(x,\bar{\lambda})$ is strongly convex in $x$, it follows that $\bar{x}=x^*$.
\end{proof}

To establish the strong variational sufficiency, we shall prove a local strong convexity property of the Moreau envelope \(M_g^\mu\) on the orthogonal complement of an appropriate subspace through Lemmas~\ref{lem:general proj}--\ref{lem:diff-ri} and Proposition~\ref{prop:Hessian}.

\begin{lemma}\label{lem:general proj}
Let  $g:\mathbb{R}^n\to\overline{\mathbb{R}}$ be a polyhedral convex function, and $\mu>0$. If $y\in\mathrm{dom}~g$, $\lambda\in\partial g(y)$, then for all $\rho_0>0$, there exists $\varepsilon>0$, such that if $\|w\|<{\rho_0}$, $0\leq t<\varepsilon$, 
    \begin{align*}
       \nabla M^{\mu}_{g}(y+\mu\lambda+tw)=\lambda+\frac{t}{\mu}\Pi_{T_{\partial {g}(y)}(\lambda)}[w]. 
    \end{align*}
\end{lemma}
\begin{proof}
Suppose that 
\[ {g}(y)=\max_{1\leq i\leq p}\{a_i^{\top}y+b_i\}+\delta_{\mathcal{P}}(y),\]
where $\mathcal{P}:=\{y\in\mathbb{R}^n\mid By\leq c\}$ with some $B \in \mathbb{R}^{q\times n}$, $c\in\mathbb{R}^q$.
Hence, for $y\in \mathcal{P}$,  
\begin{align}\label{exp-partial g}
    \partial {g}(y)=\mathrm{conv}\{a_j\mid j\in J(y)\}+\partial \delta_{\mathcal{P}}(y), 
\end{align}
where $J(y)=\{ i\mid a_i^{\top}y+b_i={g}(y)\}$. Moreover, for given ${\rho_0}>0$,  there exists $\varepsilon_{\rho_0}>0$, such that $y+tw\in \mathcal{P}$ and 
\[{g}(y+tw)>a_i^{\top}(y+tw)+b_i,~i\notin J(y) \] 
when $w\in T_{\mathcal{P}}(y)$, $\|w\|<{\rho_0}$, $t<\varepsilon_{\rho_0}$. Note that $N_{\mathcal{P}}(y)=\partial \delta_{\mathcal{P}}(y)$, then  
\[ \delta_{\mathcal{P}}(y+tw)=\sup_{v\in N_{\mathcal{P}}(y)} v^{\top}(tw)=\sup_{v\in \partial \delta_{\mathcal{P}}(y)} v^{\top}(tw). \]
This and \eqref{exp-partial g} yield that for any $w\in\mathbb{R}^n$, $\|w\|<{\rho_0}$, $t<\varepsilon_{\rho_0}$, 
      \begin{equation}\label{prop:g-ineq}
          \begin{aligned}
          {g}(y+tw)&=\max_{j\in J(y)} a_j^{\top}(y+tw)+b_j+\delta_{\mathcal{P}}(y+tw)\\   
          &=[\max_{j\in J(y)} ~ta_j^{\top}w+\delta_{\mathcal{P}}(y+tw)]+{g}(y)\\
          &=[\max_{\widetilde{\lambda}\in\partial {g}(y)} ~t\widetilde{\lambda}^{\top}w]+{g}(y).
      \end{aligned}
      \end{equation}
On the other hand, for $\lambda\in\partial {g}(y)$, from \eqref{eq:diff-tan},  there exists $\hat{\varepsilon}_{\rho_0}>0$, such that if $\|w\|<{\rho_0}$ and $0\leq \hat{t}<\hat{\varepsilon}_{\rho_0}$, 
      \begin{align}\label{belong-1}
           \lambda+\hat{t}\Pi_{T_{\partial {g}(y)}(\lambda)}[w]\in\partial {g}(y).
      \end{align}
Let $w$ be such that $\|w\|<{\rho_0}$, and let $\tilde{w}:=w-\Pi_{T_{\partial {g}(y)}(\lambda)}[w]$. Then, $\|\tilde{w}\|\leq \|w\|$ and for any $\bar{\lambda}\in \partial {g}(y)$, 
\begin{align*}
    \tilde{w}^{\top}(\bar{\lambda}-\lambda-\hat{t}\Pi_{T_{\partial {g}(y)}(\lambda)}[w])\leq 0,
\end{align*}
which comes from \eqref{convex-proj} with $\bar{\lambda}-\lambda\in T_{\partial g(y)}(\lambda)$. 
Combining this, \eqref{prop:g-ineq} and \eqref{belong-1}, 
\begin{align}
          {g}(y+t\tilde{w})=t(\lambda+\hat{t}\Pi_{T_{\partial {g}(y)}(\lambda)}[w])^{\top}\tilde{w}+{g}(y),
      \end{align}
      which (by the convexity of $g$) implies that 
      \begin{align}\label{belong2}
    \lambda+\hat{t}\Pi_{T_{\partial {g}(y)}(\lambda)}[w]\in\partial {g}(y+{t}\tilde{w}).
      \end{align} 
      Let $t:=\hat{t}\mu$ with $\hat{t}\leq \min\{\hat{\varepsilon}_{\rho_0},{\varepsilon}_{\rho_0}/\mu\}$. Since $\tilde{w}:=w-\Pi_{T_{\partial {g}(y)}(\lambda)}[w]$, 
      by \eqref{eq:partial} and \eqref{belong2}, we have  
\[ \nabla M^{\mu}_{g}(y+\mu\lambda+tw)
        =\lambda+\frac{t}{\mu}\Pi_{T_{\partial {g}(y)}(\lambda)}[w].\]
\end{proof}
The previous lemma leads to a characterization of points where 
$M^{\mu}_g$ is twice differentiable.

\begin{lemma}\label{lem:diff-ri}
   Let $g:\mathbb{R}^n\to\overline{\mathbb{R}}$ be a polyhedral convex function. Assume that $y\in\mathrm{dom}~g$, $\lambda\in\partial~g(y)$. Then the following statements hold:
   \begin{enumerate}
       \item [(i)] if $\lambda\in\mathrm{ri}~\partial g(y)$, then for any $\mu>0$, it holds that \begin{align}\label{eq:sec-diff}
        \nabla^2 M^{\mu}_g(y+\mu\lambda)=\frac{1}{\mu}\Pi_{V_{\partial g(y)}},
    \end{align}
    \item[(ii)] if there exists $\mu>0$ such that \eqref{eq:sec-diff} holds, then $\lambda\in\mathrm{ri}~\partial g(y)$.
   \end{enumerate} 
\end{lemma}
\begin{proof}
We first prove item (i). Since $\lambda\in\mathrm{ri}\,\partial g(y)$, $T_{\partial g(y)}(\lambda)=V_{\partial g(y)}$ is a linear subspace, so the projection $\Pi_{T_{\partial g(y)}(\lambda)}$ is linear. Fix any $\mu>0$. By Lemma~\ref{lem:general proj}, there exists a locally uniform limit
\begin{equation}\label{eq:limit-deriv}
\lim_{t\downarrow 0}\frac{\nabla M^{\mu}_g(y+\mu\lambda+tw)-\nabla M^{\mu}_g(y+\mu\lambda)}{t}
= \frac{1}{\mu}\Pi_{T_{\partial g(y)}(\lambda)}[w].
\end{equation}
Since the limit map is linear and the convergence is locally uniform, $\nabla M^{\mu}_g$ is differentiable at $y+\mu\lambda$ and its derivative equals this linear map, i.e.,
\[
\nabla^2 M^{\mu}_g(y+\mu\lambda)=\frac{1}{\mu}\Pi_{V_{\partial g(y)}}.
\]
This establishes item (i).

Next, we prove item (ii).  Since \eqref{eq:sec-diff} holds for some $\mu>0$, $\nabla M^{\mu}_g$ is differentiable, which implies that the directional limit in \eqref{eq:limit-deriv} must coincide with the derivative, yielding
\[
\frac{1}{\mu}\Pi_{T_{\partial g(y)}(\lambda)}[w] = \frac{1}{\mu}\Pi_{V_{\partial g(y)}}[w] \quad \forall w.
\]
Thus $\Pi_{T_{\partial g(y)}(\lambda)} = \Pi_{V_{\partial g(y)}}$, which forces \[T_{\partial g(y)}(\lambda) = V_{\partial g(y)},\] and therefore \mbox{$\lambda\in\mathrm{ri}\,\partial g(y)$}. This completes the proof.
\end{proof}

Using Lemma \ref{lem:diff-ri} and the Moreau decomposition (Proposition  \ref{Moreau-decom}), i.e., for any $w\in\mathbb{R}^n$,  \mbox{$w=y+\mu\lambda$} with $y=\mathrm{prox}_{\mu g}(w)$ and $\lambda\in\partial g(y)$, we now derive a lower bound of  $\nabla^2 M^{\mu}_g(w)$ if it exists.

\begin{proposition} \label{prop:Hessian}
 Suppose that $g:\mathbb{R}^m\to\overline{\mathbb{R}}$ is a polyhedral convex function and $\mu>0$. Let $\bar{w}\in\mathbb{R}^m$ be given, admitting the decomposition $\bar{w}={\bar{y}}+ \mu{\bar{\lambda}}$ where ${\bar{y}}=\mathrm{prox}_{\mu g}(\bar{w})$ and ${\bar{\lambda}}\in\partial g({\bar{y}})$. Then, there exists a convex neighborhood $U$ of $\bar{w}$ such that for any $w\in U$, if $\nabla M^{\mu}_g$ is differentiable at $w$, then the following inequality holds:
\begin{align}
     \nabla^2 M^{\mu}_g(w)\succeq \frac{1}{\mu}(I-\Pi_{\mathcal{S}(\bar{y}\mid \bar{\lambda})}),
 \end{align}
where \[\mathcal{S}(\bar{y}\mid \bar{\lambda})=T_g(\bar{y}\mid \bar{\lambda})-T_g(\bar{y}\mid \bar{\lambda})\]
 is the smallest subspace containing the tangent cone $T_g(\bar{y}\mid \bar{\lambda})$.

\end{proposition} 
\begin{proof}
    Since $g^*$ is also a polyhedral convex function \cite[Theorem~19.2]{rockafellar1997convex}, by \eqref{eq:one-sided derivate}, there exists a convex neighborhood $\Lambda$ of $\bar{\lambda}$ such that 
\begin{align}\label{prop:include1}
    \partial g^*(\lambda)\subseteq \partial g^*(\bar{\lambda}),~\forall~ \lambda\in\Lambda.
\end{align}  
In addition, we have
\begin{align*}
        T_{\partial g^*(\lambda)}(y)&\subseteq T_{\partial g^*(\bar{\lambda})}(y)\subseteq 
        T_{\partial g^*(\bar{\lambda})}(\bar{y})-T_{\partial g^*(\bar{\lambda})}(\bar{y})=V_{\partial g^*(\bar{\lambda})},
    \end{align*}
which by \eqref{eq:Tg} can be equivalently reformulated as
\begin{align}\label{eq:include}
        T_g(y\mid \lambda)\subseteq T_g(y\mid\bar{\lambda})\subseteq \mathcal{S}(\bar{y}\mid \bar{\lambda})=V_{\partial g^*(\bar{\lambda})}.
    \end{align}
    
Note that the mapping
\[
	\lambda(w):=
	\frac{1}{\mu}\operatorname{prox}_{(\mu g)^*}(w)
\]
is continuous, and
\(\lambda(\bar w)=\bar\lambda\in\Lambda\).
Hence, there exists a convex neighborhood
\(\mathcal U\) of \(\bar w\) such that
\(\lambda(w)\in\Lambda\) for every \(w\in\mathcal U\).    
 By Proposition  \ref{Moreau-decom}, this together with \eqref{eq:partial} and \eqref{prop:include1} yields that \[y(w)\in \partial g^*(\lambda(w))\subseteq \partial g^*(\bar{\lambda}),\] where $y(w)=\mathrm{prox}_{\mu g}(w)=w-\mu\lambda(w)$.
Invoking Lemma \ref{lem:diff-ri},  if $\nabla M^{\mu}_g$ is differentiable at $w$, then
\[T_{\partial g(y(w))}(\lambda(w))=V_{\partial g(y(w))},\]
and $\nabla^2 M^{\mu}_g(w)=\frac{1}{\mu}\Pi_{V_{\partial g(y(w))}}$.
Therefore, combining this with \eqref{polar} and \eqref{eq:Tg}, we have   
\[\nabla^2 M^{\mu}_g(w)=\frac{1}{\mu}(I-\Pi_{T_g(y(w)\mid \lambda(w))}).\]
 Incorporating this with the inclusion relation \eqref{eq:include} completes the proof.
\end{proof}

To illustrate the local strong convexity property in
Proposition~\ref{prop:Hessian}, consider the absolute value function
\(g(t)=\vert t\vert\) and let \(\mu>0\). Its proximal mapping is
\[
	\operatorname{prox}_{\mu g}(w)
	=
	\operatorname{sign}(w)\max\{\vert w\vert-\mu,0\},
\]
while its Moreau envelope is given by
\[
	M_g^\mu(w)
	=
	\begin{cases}
		\dfrac{1}{2\mu}w^2, & \vert w\vert\leq\mu,\\[1mm]
		\vert w\vert-\dfrac{\mu}{2}, & \vert w\vert>\mu.
	\end{cases}
\]
Hence, \(M_g^\mu\) is twice differentiable on
\((-\infty,-\mu)\), \((-\mu,\mu)\), and \((\mu,\infty)\), with
\[
	\nabla^2 M_g^\mu(w)
	=
	\begin{cases}
		\dfrac{1}{\mu}, & \vert w\vert<\mu,\\[1mm]
		0, & \vert w\vert>\mu.
	\end{cases}
\]

For a given \(\bar w\in\mathbb R\), two cases arise. First, if
\(\vert \bar w \vert<\mu\), then
\[
	\bar y=0,\qquad
	\bar\lambda=\frac{\bar w}{\mu},\qquad
	T_g(\bar y\mid\bar\lambda)
	=
	S(\bar y\mid\bar\lambda)
	=
	\{0\}.
\]
Consequently,
\(S(\bar y\mid\bar\lambda)^\perp=\mathbb R\), and
Proposition~\ref{prop:Hessian} implies that
\(\nabla^2M_g^\mu(w)\geq 1/\mu\) in a neighborhood of
\(\bar w\). Thus, \(M_g^\mu\) is locally
\(1/\mu\)-strongly convex near \(\bar w\), consistently with the
explicit Hessian formula above.

Second, if \(\vert \bar w \vert\geq\mu\), then
\[
	\bar y
	=
	\bar w-\operatorname{sign}(\bar w)\mu,
	\qquad
	\bar\lambda
	=
	\operatorname{sign}(\bar w),
\]
and
\[
	S(\bar y\mid\bar\lambda)=\mathbb R.
\]
Indeed, \(T_g(\bar y\mid\bar\lambda)\) is a half-axis when
\(\vert \bar w \vert=\mu\) and equals \(\mathbb R\) when
\(\vert \bar w \vert>\mu\). Therefore,
\(S(\bar y\mid\bar\lambda)^\perp=\{0\}\), so that the local
strong convexity property in Proposition~\ref{prop:Hessian} is degenerate in this
case and yields only
\[
	\nabla^2M_g^\mu(w)\geq0
\]
wherever the Hessian exists in a neighborhood of \(\bar w\).
This is again consistent with the explicit formula: the Moreau
envelope is affine when \(\vert w\vert>\mu\), while at
\(\vert w\vert=\mu\) its Hessian does not exist.

\begin{assumption}\label{assp:ADMM}
~
\begin{enumerate}
    \item[(i)] The pair $(\bar{x},\bar{\lambda})$ satisfies the first-order optimal condition \eqref{first order condition1}.
    \item[(ii)] There exists a convex neighborhood $U_{\bar{x}}$ of $\bar{x}$ on which ${\nabla}^2 f({x})$ exists and is continuous. Moreover, $\nabla^2 f(\bar{x})$ is positive definite relative to the subspace $V_{\mathcal{C}}\cap \mathcal{S}_A(\bar{x}\mid \bar{\lambda})$ where \[\mathcal{S}_A(\bar{x}\mid \bar{\lambda})=\{v\mid Av\in \mathcal{S}(A\bar{x}\mid\bar{\lambda})\}.\]
\end{enumerate}
\end{assumption}
Next, by leveraging Proposition \ref{prop:Hessian}, we derive that the strong variational sufficiency for Problem~\eqref{prob:com} holds with respect to $(\bar{x},\bar{\lambda})$ satisfying Assumption~\ref{assp:ADMM}. For the remainder of this paper, set
\[
	P:=I-\Pi_{\mathcal S(A\bar x\mid\bar\lambda)}.
\]
\begin{proposition}\label{prop:local duality}
     Assume that $(\bar{x}, \bar{\lambda})$ satisfies Assumption~\ref{assp:ADMM}. Then, the strong variational sufficiency for Problem~\eqref{prob:com} holds with respect to $(\bar{x},\bar{\lambda})$.
\end{proposition}
\begin{proof}
For any nonzero \(v\in V_{\mathcal C}\) satisfying
\(v^\top(-A^\top PA)v\geq0\), we have \(PAv=0\), and hence
\[  v\in
	V_{\mathcal C}\cap
	\mathcal S_A(\bar x\mid\bar\lambda).
\]
Thus, Assumption~\ref{assp:ADMM}(ii) gives
\(v^\top\nabla^2f(\bar x)v>0\). Applying Finsler's lemma
\cite[Lemma~1.2]{Cimpric2015} to the corresponding quadratic
forms restricted to \(V_{\mathcal C}\), there exists
\(\beta_0>0\) such that
\[
	\nabla^2f(\bar x)+\beta_0A^\top PA
\]
is positive definite relative to \(V_{\mathcal C}\). Hence, by the continuity of $\nabla^2 f$, there exists a convex neighborhood $\mathcal{X}_1$ of $\bar{x}$ on which
\begin{align}\label{def-Psi}
    \Psi_{\beta_0}(x):=f(x)+\frac{\beta_0}{2} x^{\top}A^{\top}PAx+\delta_{\mathcal{C}}(x)
\end{align}
is strongly convex. 

On the other hand, fix \(\beta\geq\beta_0\) and set 
$\bar w_\beta:=A\bar x+\bar\lambda/\beta.$
By Proposition~\ref{prop:Hessian}, there exists a convex neighborhood
\(\mathcal U_\beta\) of \(\bar w_\beta\) such that $\nabla^2M_g^{1/\beta}(w)\succeq\beta P$
whenever \(w\in\mathcal U_\beta\) and
\(\nabla M_g^{1/\beta}\) is differentiable at \(w\).
Shrinking \(\mathcal U_\beta\) if necessary, we may assume that
it is open and convex. Define
\[
	\widehat\Phi_\beta(w)
	:=
	M_g^{1/\beta}(w)-\frac{\beta}{2}w^\top Pw.
\]
Since \(\nabla M_g^{1/\beta}\) is Lipschitz continuous,
\(\nabla\widehat\Phi_\beta\) is locally Lipschitz continuous.
Moreover, wherever its Hessian exists in \(\mathcal U_\beta\),
\[
	\nabla^2\widehat\Phi_\beta(w)
	=
	\nabla^2M_g^{1/\beta}(w)-\beta P
	\succeq0.
\]
It follows from the Hessian characterization of convexity
recalled in Section~\ref{sec:Pre} that \(\widehat\Phi_\beta\) is convex on
\(\mathcal U_\beta\).

Let 
\begin{align}\label{def-Phi}
    \Phi_\beta(x,\lambda)
	:=
	M_g^{1/\beta}
	\left(Ax+\frac{\lambda}{\beta}\right)
	-\frac{\beta}{2}x^\top A^\top PAx.
\end{align}
Let \(\mathcal X\) and
\(\Lambda_\beta\) be  convex neighborhoods of \(\bar x\) and \(\bar\lambda\), respectively, such that $Ax+\lambda/\beta\in\mathcal{U}_{\beta}$ for all $(x,\lambda)\in\mathcal{X}\times \Lambda_{\beta}$. Then,  $\Phi_\beta(x,\lambda)$
is convex in \(x\) on \(\mathcal X\) for every
\(\lambda\in\Lambda_\beta\), since
\[
	\Phi_\beta(x,\lambda)
	=
	\widehat\Phi_\beta
	\left(Ax+\frac{\lambda}{\beta}\right)
	+\left\langle A^\top P\lambda,x\right\rangle
	+\frac{1}{2\beta}\lambda^\top P\lambda.
\]
It follows that for any $\lambda\in\Lambda_{\beta}$, $l_{\beta}(x,\lambda)$ is strongly convex on $\mathcal{X}_{\beta}:=\mathcal{X}_1\cap\mathcal{X}$,
since 
\[l_{\beta}(x,\lambda)=\Phi_\beta(x,\lambda)+\Psi_{\beta_0}(x)+\frac{\beta-\beta_0}{2}x^\top A^\top PAx-\frac{1}{2\beta}\|\lambda\|^2.\]
Because
\begin{align*}
    M_g^{1/\beta}(Ax+\frac{\lambda}{\beta})-\frac{1}{2\beta}\|\lambda\|^2&=\inf_{y} g(y)+\langle \lambda, Ax-y\rangle+\frac{\beta}{2}\|Ax-y\|^2,
\end{align*}
   $l_{\beta}(x,\lambda)$ is concave on $\Lambda_{\beta}$ for any $x\in\mathcal{X}_{\beta}\cap\mathcal{C}$. This completes the proof.
\end{proof}

\section{Local Convergence Analysis}\label{sec:local}
In this section, we establish the local convergence of ADMM scheme \eqref{alg:classical ADMM1} for Problem \eqref{prob:linear-eq} under Assumption~\ref{assp:ADMM}. Our analysis proceeds in two main steps. First, we derive a key inequality whose structure is similar to those obtained in the analysis of classical ADMM for convex problems. Subsequently, building upon this fundamental inequality and local duality established in Proposition \ref{prop:local duality}, we prove the local convergence of the sequence generated by ADMM scheme \eqref{alg:classical ADMM1}.

\begin{proposition}\label{prop:fund}
    Let $(\bar{x}, \bar{\lambda})$ satisfy Assumption~\ref{assp:ADMM}. Then, there exists $\beta_0>0$ such that for any $\beta\geq \beta_0$, there exists a convex neighborhood $\mathcal{X}_{\beta}\times\Lambda_{\beta}$ of $(\bar{x}, \bar{\lambda})$ with the following property: 
    if $y^k\in\mathrm{dom}~g$, $\lambda^k\in\Lambda_{\beta}\cap\partial g(y^k)$ and $x^{k+1}$ generated by \eqref{alg:cl1-ADMM1} lies in $\mathcal{X}_{\beta}$,
      then the subsequent iterate $(x^{k+1},y^{k+1},\lambda^{k+1})$ generated by  \eqref{alg:classical ADMM1} satisfies
    \begin{equation}
        \begin{aligned}\label{ineq:fund}
    &l_{\beta_0}(x^{k+1},\lambda^*)-l_{\beta_0}(\bar{x},{\lambda}^*)+\frac{\theta\beta-{\beta_0}}{2}(x^{k+1}-\bar{x})^{\top}A^{\top}PA(x^{k+1}-\bar{x})\\
    \leq &\frac{\beta}{2}(\|y^k-\bar{y}\|^2-\|y^{k+1}-\bar{y}\|^2)+\frac{1}{2\beta}(\|\lambda^k-{{\lambda}^*}\|^2-\|\lambda^{k+1}-{{\lambda}^*}\|^2)\\
    &-\frac{\beta}{2}\|y^{k+1}-y^k\|^2-\frac{(1-\theta)\beta-{\beta_0}}{2\beta^2}\|\lambda^{k+1}-\lambda^k\|^2,
\end{aligned}
    \end{equation}
    where $(\bar{x},\lambda^*)$ is any saddle point of $l_{\beta_0}(x,\lambda)$ defined in \eqref{eq:reduced-lag} relative to $\mathcal{X}_{\beta}\times\Lambda_{\beta}$,  $\theta\in (0,1)$ and $\bar{y}=A\bar{x}$.
\end{proposition}

\begin{proof}
By Proposition \ref{prop:local duality} with \(\beta=\beta_0\) and
	intersecting the resulting convex neighborhood with that constructed
	for the current \(\beta\geq\beta_0\), we may choose a convex
	neighborhood \(\mathcal X_\beta\times\Lambda_\beta\) of
	\((\bar x,\bar\lambda)\) such that
	\(\ell_{\beta_0}(\cdot,\lambda)\) is strongly convex on
	\(\mathcal X_\beta\) for every \(\lambda\in\Lambda_\beta\),
	\(\ell_{\beta_0}(x,\cdot)\) is concave on \(\Lambda_\beta\)
	for every \(x\in\mathcal X_\beta\cap\mathcal C\),
	\(\Psi_{\beta_0}\) defined in~\eqref{def-Psi} is strongly convex on
	\(\mathcal X_\beta\), and
	\(\Phi_\beta(\cdot,\lambda)\) defined in~\eqref{def-Phi} is convex on
	\(\mathcal X_\beta\) for every \(\lambda\in\Lambda_\beta\).
   
   By the optimality condition of \eqref{alg:cl1-ADMM1} and the strong convexity of $\Psi_{\beta_0}+\langle \lambda^k, Ax\rangle$, we have
    \begin{align*}
    &f(x^{k+1})+\langle \lambda^k, Ax^{k+1}\rangle+\frac{\beta_0}{2}{x^{k+1}}^{\top}A^{\top}PAx^{k+1}\\
    \leq &f(\bar{x})+\langle \lambda^k,A\bar{x}\rangle+\frac{\beta_0}{2}\bar{x}^{\top}A^{\top}PA\bar{x}\\
    &+\langle -{\beta_0} A^{\top}PAx^{k+1}+\beta A^{\top}(Ax^{k+1}-y^k),\bar{x}-x^{k+1}\rangle.
\end{align*}
It can be reformulated as
\begin{equation}\label{ineq:x}
    \begin{aligned}
        &f(x^{k+1})+\langle \lambda^k, Ax^{k+1}\rangle-\frac{\beta_0}{2}(x^{k+1}-\bar{x})^{\top}A^{\top}PA(x^{k+1}-\bar{x})\\
    \leq &f(\bar{x})+\langle \lambda^k,A\bar{x}\rangle+\frac{\beta}{2}\|A\bar{x}-y^k\|^2-\frac{\beta}{2}\|Ax^{k+1}-y^k\|^2-\frac{\beta}{2}\|Ax^{k+1}-A\bar{x}\|^2.
\end{aligned}
\end{equation}
 Let $\bar{y}:=A\bar{x}$.  By the optimality condition of \eqref{alg:cl1-ADMM2} and the convexity of $g$, we obtain
 \begin{equation}\label{ineq:y1}
     \begin{aligned}
   & g(y^{k+1})+\langle\lambda^k,-y^{k+1}\rangle+\frac{\beta}{2}\|Ax^{k+1}-y^{k+1}\|^2\\
    \leq& g(\bar{y})+\langle\lambda^k,-\bar{y}\rangle+\frac{\beta}{2}\|Ax^{k+1}-\bar{y}\|^2-\frac{\beta}{2}\|y^{k+1}-\bar{y}\|^2.
    \end{aligned}
 \end{equation}
On the other hand, by \eqref{alg:cl1-ADMM2} and \eqref{alg:cl1-ADMM3}, we have 
\begin{equation}\label{eq:moreau-expan}
    \begin{aligned}
    &g(y^{k+1})+\langle \lambda^k, Ax^{k+1}-y^{k+1}\rangle+\frac{\beta}{2}\|Ax^{k+1}-y^{k+1}\|^2\\
=&M^{1/\beta}_g(Ax^{k+1}+\frac{\lambda^k}{\beta})-\frac{1}{2\beta}\|\lambda^k\|^2,\end{aligned}
\end{equation}
and
\begin{align}\label{eq:moreau-grad}
    \lambda^{k+1}&=\nabla M^{1/\beta}_g(Ax^{k+1}+\frac{\lambda^k}{\beta})\in\partial g(y^{k+1}).
\end{align}
Since $\Phi_{\beta}(x,\lambda^k)$ is convex on $\mathcal{X}_{\beta}$ and $x^{k+1}\in\mathcal{X}_{\beta}$, using \eqref{eq:moreau-grad} yields that
\begin{align}\label{ineq:Moreau-P}
   \Phi_{\beta}(x^{k+1},\lambda^k)-\Phi_{\beta}(\bar{x},\lambda^k)
    \leq  -\langle A^{\top}\lambda^{k+1}-\beta A^{\top}PAx^{k+1}, \bar{x}-x^{k+1}\rangle.
\end{align}
By combining formulas \eqref{eq:moreau-expan} and \eqref{ineq:Moreau-P} with 
\begin{align*}
    &M^{1/\beta}_g(A\bar{x}+\frac{\lambda^k}{\beta})-\frac{1}{2\beta}\|\lambda^k\|^2
\leq g(\bar{y})+\langle \lambda^k, A\bar{x}-\bar{y}\rangle+\frac{\beta}{2}\|A\bar{x}-\bar{y}\|^2,
\end{align*}
we further obtain
\begin{equation}\label{ineq:y2}
    \begin{aligned}
    &g(y^{k+1})+\langle\lambda^k,-y^{k+1}\rangle+\frac{\beta}{2}\|Ax^{k+1}-y^{k+1}\|^2+\frac{\beta}{2}(x^{k+1}-\bar{x})^{\top}A^{\top}PA(x^{k+1}-\bar{x})\\
    \leq &g(\bar{y})+\langle\lambda^k,-\bar{y}\rangle+\frac{\beta}{2}\|Ax^{k+1}-\bar{y}\|^2-\frac{\beta}{2}\|y^{k+1}-\bar{y}\|^2+\frac{1}{2\beta}\|\lambda^{k+1}-\lambda^k\|^2.
\end{aligned}
\end{equation}
As $g$ is convex, $\lambda^k\in\partial g(y^k)$, and $\lambda^{k+1}\in\partial g(y^{k+1})$, we have
\[\|Ax^{k+1}-y^k\|^2\geq \frac{1}{\beta^2}\|\lambda^{k+1}-\lambda^k\|^2+\|y^{k+1}-y^k\|^2.\]
Furthermore,  multiplying both sides of equation
\eqref{ineq:y1} by $(1-\theta)$ and both sides of equation \eqref{ineq:y2} by $\theta$ with $\theta\in(0,1)$,  and adding them with equation \eqref{ineq:x}, we obtain
\begin{equation}\label{ineq:entire1}
    \begin{aligned}   &L_\beta(x^{k+1},y^{k+1},\lambda^k)+\frac{\theta\beta-{\beta_0}}{2}(x^{k+1}-\bar{x})^{\top}A^{\top}PA(x^{k+1}-\bar{x})\\
    \leq &f(\bar{x})+g(\bar{y})+\frac{\beta}{2}(\|y^k-\bar{y}\|^2-\|y^{k+1}-\bar{y}\|^2)\\
    &-\frac{1-\theta}{2\beta}\|\lambda^{k+1}-\lambda^k\|^2-\frac{\beta}{2}\|y^{k+1}-y^k\|^2.
\end{aligned}
\end{equation}
By the equivalence between  \eqref{first order condition1} and \eqref{first order condition2}, the set of saddle points of $l_{\beta_0}$ relative to $\mathcal{X}_{\beta}\times\Lambda_{\beta}$ is nonempty.
Let $(\bar{x},{\lambda}^*)$ be an arbitrary saddle point in this set. Then, we have
\begin{align}\label{eq:saddle}
    l_{\beta_0}(\bar{x},{{\lambda}^*})=f(\bar{x})+g(A\bar{x})=f(\bar{x})+g(\bar{y}).
\end{align}
Furthermore, 
\begin{align*}
    l_{\beta_0}(x^{k+1},{{\lambda}^*})\leq & f(x^{k+1})+g(y^{k+1})+\langle{{\lambda}^*},Ax^{k+1}-y^{k+1}\rangle+\frac{\beta_0}{2}\|Ax^{k+1}-y^{k+1}\|^2\\
    =&L_{\beta}(x^{k+1},y^{k+1},\lambda^k)+\frac{\beta_0}{2\beta^2}\|\lambda^{k+1}-\lambda^k\|^2+\frac{1}{2\beta}(\|\lambda^k-\lambda^*\|^2-\|\lambda^{k+1}-\lambda^*\|^2).
\end{align*}
Combining this with \eqref{ineq:entire1} and \eqref{eq:saddle}, we obtain \eqref{ineq:fund}.
\end{proof}

 Based on Proposition \ref{prop:fund}, we now present the local convergence result for the ADMM scheme \eqref{alg:classical ADMM1} applied to Problem \eqref{prob:linear-eq}.
\begin{proposition}\label{prop:local-convergence}
	Assume that \((\bar x,\bar\lambda)\) satisfies Assumption~\ref{assp:ADMM}.
	Then, there exists \(\beta_0>0\) such that, for every
	\(\beta>2\beta_0\), there exist closed convex neighborhoods
	\(\mathcal X_\beta\) and \(\Lambda_\beta\) of
	\(\bar x\) and \(\bar\lambda\), respectively, and a constant
	\(\rho_\beta>0\) such that
		$\mathbb B(\bar\lambda,\rho_\beta)
		\subseteq \operatorname{int}\Lambda_\beta$
	with the following property. Let
	\(\{(x^k,y^k,\lambda^k):k\in\mathbb N\}\)
	be a well-defined sequence generated by ADMM scheme~\eqref{alg:classical ADMM1}.
	Suppose that
	\begin{equation}
		\lambda^0\in\partial g(y^0),
		\qquad
		\beta^2\|y^0-A\bar x\|^2
		+\|\lambda^0-\bar\lambda\|^2
		\leq \rho_\beta^2,
		\label{the:initial}
	\end{equation}
	and that $x^k\in\mathcal X_\beta$, for all $k\geq1$.
	Then
	\[
		x^k\to\bar x,\qquad
		y^k\to A\bar x,\qquad
		\lambda^k\to\widehat \lambda
	\]
	for some \(\widehat \lambda\in\Lambda_\beta^*\), where
	\[
		\Lambda_\beta^*
		:=
		\left\{
		\lambda\in\Lambda_\beta
		\;\Big\vert \;
		\begin{array}{l}
			0\in\nabla f(\bar x)+A^\top\lambda
			+N_{\mathcal C}(\bar x),\\[-1mm]
			\lambda\in\partial g(A\bar x)
		\end{array}
		\right\}.
	\]
	Moreover, $\widehat \lambda\in\mathbb B(\bar\lambda,\rho_\beta)
		\subseteq\operatorname{int}\Lambda_\beta$.
	
\end{proposition}
\begin{proof}
	By Propositions~\ref{prop:local duality} and~\ref{prop:fund}, there exists $\beta_0>0$ such that, for every $\beta>2\beta_0$, there exist closed convex neighborhoods $\mathcal{X}_{\beta}$ of $\bar{x}$ and $\Lambda_{\beta}$ of $\bar{\lambda}$, respectively, satisfying the property stated in Proposition~\ref{prop:fund}. Since $\Lambda_\beta$ is a convex neighborhood of $\bar\lambda$, one can choose $\rho_\beta>0$ so that $\mathbb B(\bar\lambda,\rho_\beta)
		\subseteq\operatorname{int}\Lambda_\beta$.
    
	We first show that the assumptions of Proposition~\ref{prop:fund} remain valid at every iteration. Set $\bar y:=A\bar x$. Since $(\bar x,\bar\lambda)$ satisfies~\eqref{first order condition1}, the pair $(\bar x,\bar\lambda)$ is a saddle point of $\ell_{\beta_0}$ relative to $\mathcal X_\beta\times\Lambda_\beta$. Applying Proposition~\ref{prop:fund} with $\theta=1/2$ and $\lambda^*=\bar\lambda$ gives
	\[
	\begin{aligned}
		\beta^2\|y^{k+1}-\bar y\|^2
		+\|\lambda^{k+1}-\bar\lambda\|^2\leq
		\beta^2\|y^k-\bar y\|^2
		+\|\lambda^k-\bar\lambda\|^2.
	\end{aligned}
	\]
	Hence, by~\eqref{the:initial},
	\[
		\lambda^k\in
		\mathbb B(\bar\lambda,\rho_\beta)
		\subseteq\operatorname{int}\Lambda_\beta,
		\qquad k\in\mathbb N.
	\]
	Moreover, the optimality condition of the $y$-subproblem and~\eqref{alg:cl1-ADMM3} imply
	\[
		\lambda^{k+1}\in\partial g(y^{k+1}).
	\]
	Therefore, starting from $\lambda^0\in\partial g(y^0)$ and using the assumption $x^{k+1}\in\mathcal X_\beta$, Proposition~\ref{prop:fund} applies inductively for all $k\in\mathbb N$.

		By the construction in the proof of Proposition~\ref{prop:local duality}, there exists a constant $\sigma>0$, independent of $\lambda\in\Lambda_\beta$, such that $\ell_{\beta_0}(\cdot,\lambda)$ is $\sigma$-strongly convex on $\mathcal X_\beta$. For any $\lambda^*\in\Lambda_\beta^*$, the equivalence of \eqref{first order condition1} and \eqref{first order condition2} implies that $(\bar x,\lambda^*)$ is a saddle point of $\ell_{\beta_0}$ relative to $\mathcal X_\beta\times\Lambda_\beta$.  Consequently,
	\[
		\ell_{\beta_0}(x, \lambda^*)
		-\ell_{\beta_0}(\bar x, \lambda^*)
		\geq
		\frac{\sigma}{2}\|x-\bar x\|^2,
		\qquad x\in\mathcal X_\beta.
	\]
	Using Proposition~\ref{prop:fund} with $\theta=1/2$, we obtain
	\begin{equation}
	\begin{aligned}
		\frac{\sigma}{2}\|x^{k+1}-\bar x\|^2
		\leq{}&
		\frac{\beta}{2}
		\bigl(
			\|y^k-\bar y\|^2
			-\|y^{k+1}-\bar y\|^2
		\bigr)\\
		&+\frac{1}{2\beta}
		\bigl(
			\|\lambda^k- \lambda^*\|^2
			-\|\lambda^{k+1}- \lambda^*\|^2
		\bigr)\\
		&-\frac{\beta}{2}\|y^{k+1}-y^k\|^2
		-\frac{\beta-2\beta_0}{4\beta^2}
		\|\lambda^{k+1}-\lambda^k\|^2.
	\end{aligned}
	\label{ineq-fund-2}
	\end{equation}
	Summing~\eqref{ineq-fund-2} from $k=0$ to $K$ yields
	\begin{equation}
	\begin{aligned}
		&\frac{\sigma}{2}\sum_{k=0}^{K}
		\|x^{k+1}-\bar x\|^2
		+\frac{\beta}{2}\sum_{k=0}^{K}
		\|y^{k+1}-y^k\|^2
		+\frac{\beta-2\beta_0}{4\beta^2}
		\sum_{k=0}^{K}
		\|\lambda^{k+1}-\lambda^k\|^2\\
		&\leq
		\frac{\beta}{2}\|y^0-\bar y\|^2
		+\frac{1}{2\beta}
		\|\lambda^0- \lambda^*\|^2.
	\end{aligned}
	\label{eq:local-summability}
	\end{equation}
	Letting $K\to\infty$ gives
	\[
		x^k\to\bar x,\qquad
		y^{k+1}-y^k\to0,\qquad
		\lambda^{k+1}-\lambda^k\to0.
	\]
	Since
	\[
		y^{k+1}
		=
		Ax^{k+1}
		-\frac{1}{\beta}
		(\lambda^{k+1}-\lambda^k),
	\]
	we have
	\[
		\|y^{k+1}-\bar y\|
		\leq
		\|Ax^{k+1}-A\bar x\|
		+\frac{1}{\beta}\|\lambda^{k+1}-\lambda^k\|,
	\]
	which implies that $y^k\to\bar y=A\bar x$.

Since $\{\lambda^k\}$ is contained in the compact set $\mathbb B(\bar\lambda,\rho_\beta)$, it admits a cluster point $\widehat\lambda$. Let $\lambda^{k_j}\to\widehat\lambda$. Because $\lambda^{k_j+1}-\lambda^{k_j}\to0$, we also have $\lambda^{k_j+1}\to\widehat\lambda$. The closedness of the graph of $\partial g$, together with $y^{k_j+1}\to\bar y$ and $\lambda^{k_j+1}\in\partial g(y^{k_j+1})$, yields $\widehat\lambda\in\partial g(\bar y)$.

On the other hand, the optimality condition of the $x$-subproblem and~\eqref{alg:cl1-ADMM3} give
\[
    0\in
    \nabla f(x^{k+1})
    +A^\top\lambda^{k+1}
    +\beta A^\top(y^{k+1}-y^k)
    +N_{\mathcal C}(x^{k+1}).
\]
Passing to the limit along $\{k_j\}$ and using the closedness of the graph of $N_{\mathcal C}$, we obtain
\[
    0\in
    \nabla f(\bar x)
    +A^\top\widehat\lambda
    +N_{\mathcal C}(\bar x).
\]
Therefore, by the definition of \(\Lambda_\beta^*\), $\widehat\lambda\in\Lambda_\beta^*$. Finally, applying~\eqref{ineq-fund-2} with $\lambda^*=\widehat\lambda$ shows that the sequence
\[
    \frac{\beta}{2}\|y^k-\bar y\|^2
    +\frac{1}{2\beta}
    \|\lambda^k-\widehat\lambda\|^2
\]
is nonincreasing. Moreover, this sequence converges to zero
along the subsequence \(\{k_j\}\).
Hence the entire sequence converges to zero. Therefore, $\lambda^k\to\widehat\lambda$,  and the inclusion
\[
    \widehat \lambda\in
    \mathbb B(\bar\lambda,\rho_\beta)
    \subseteq\operatorname{int}\Lambda_\beta
\]
follows directly from the preceding invariance estimate.
\end{proof}

Next, we prove the local linear convergence of  the ADMM scheme \eqref{alg:classical ADMM1} for the case where $\mathcal{C}$ is a polyhedral convex set. Suppose that $(\bar{x},\bar{\lambda})$ satisfies Assumption~\ref{assp:ADMM}.  Let \(u:=(x,y,\lambda)\) and define
\begin{equation}\label{eq:T_map}
\mathfrak R(u)
    :=
    \begin{bmatrix}
        \nabla f(\bar x)
        +\nabla^2f(\bar x)(x-\bar x)
        +A^\top\lambda+N_{\mathcal C}(x)\\
        \partial g(y)-\lambda\\
        Ax-y
    \end{bmatrix}.
\end{equation}   
Then, by Corollary \ref{coro-unique-x} and Proposition \ref{prop:local duality}, the following lemma characterizes the local zero set of $\mathfrak{R}$.
\begin{lemma}\label{lem-T-solution}
    Assume that $(\bar{x}, \bar{\lambda})$ satisfies Assumption~\ref{assp:ADMM}. Let $\beta$ and $\mathcal{X}_{\beta}\times\Lambda_{\beta}\subseteq\mathbb{R}^n\times\mathbb{R}^m$ as stated in Proposition \ref{prop:local-convergence}. Then, it holds that
    \begin{align}\label{eq-local-solution}
        \mathfrak{R}^{-1}(0)\cap \mathcal{X}_{\beta}\times\mathbb{R}^m\times\Lambda_{\beta}=\{\bar{x}\}\times\{A\bar{x}\}\times \Lambda^*_{\beta}.
    \end{align}
\end{lemma}
\begin{proof}
 Let \begin{align*}
        \hat{l}_{\beta}(x,\lambda):=&\langle\nabla f(\bar{x}),x-\bar{x}\rangle+\frac{1}{2}\langle x-\bar{x},\nabla^2 f(\bar{x})(x-\bar{x})\rangle\\
        &+M^{1/\beta}_g(Ax+\frac{\lambda}{\beta})
        +\delta_{\mathcal{C}}(x)-\frac{1}{2\beta}\|\lambda\|^2.
    \end{align*}
	For \((\widehat x,\widehat y,\widehat \lambda)\in \mathfrak R^{-1}(0)\cap (\mathcal X_\beta\times\mathbb R^m\times\Lambda_\beta)\), by the definition of $\mathfrak{R}$ and relation~\eqref{eq:partial}, it holds that \(\widehat y=A\widehat x\) and 
	\[
		0\in\partial_x\widehat\ell_\beta(\widehat x,\widehat \lambda),
		\qquad
		0\in\partial_\lambda
		[-\widehat\ell_\beta](\widehat x,\widehat \lambda).
	\]
	By an argument similar to the proof of Proposition \ref{prop:local duality}, $\hat{l}_{\beta}(x,\lambda)$ is strongly convex in $x$ when $\lambda\in\Lambda_{\beta}$  and concave in $\lambda$ when $x\in\mathcal{X}_{\beta}\cap\mathcal{C}$.
	Hence, \((\widehat x,\widehat \lambda)\) is a saddle point of
	\(\widehat\ell_\beta\) relative to $\mathcal{X}_{\beta}\times\Lambda_{\beta}$. Moreover, Assumption~\ref{assp:ADMM}(i) and relation~\eqref{eq:partial} imply that
\((\bar x,\bar\lambda)\) is also a saddle point of
\(\widehat\ell_\beta\) relative to
\(\mathcal X_\beta\times\Lambda_\beta\).
Hence, the argument in Corollary~\ref{coro-unique-x} yields
\(\widehat x=\bar x\). The definition of \(\mathfrak R\) further
	yields
	\[
		\widehat \lambda\in\partial g(A\bar x),
		\qquad
		0\in\nabla f(\bar x)+A^\top \widehat \lambda
		+N_{\mathcal C}(\bar x),
	\]
	and thus \(\widehat \lambda\in\Lambda_\beta^*\). Conversely, every \(\widehat \lambda\in\Lambda_\beta^*\) satisfies
	\[
		(\bar x,A\bar x,\widehat \lambda)\in\mathfrak R^{-1}(0).
	\]
	This proves~\eqref{eq-local-solution}.
\end{proof}

By the optimality conditions of the subproblems of ADMM scheme \eqref{alg:classical ADMM1}, we obtain the estimate of $\mathrm{dist}(0, \mathfrak{R}(u^{k}))$.
\begin{lemma}\label{lem-ineq-T}
Assume that $(\bar{x}, \bar{\lambda})$ satisfies Assumption~\ref{assp:ADMM}. Let the sequence $u^{k+1}:=(x^{k+1},y^{k+1},\lambda^{k+1})$ be generated by ADMM scheme \eqref{alg:classical ADMM1} under the conditions of Proposition \ref{prop:local-convergence}.
Then, for any $\gamma>0$, there exists $K\in \mathbb{N}$ such that for all $k\geq K$, 
\begin{equation}\label{ineq-T}
    \begin{aligned}
    \mathrm{dist}^2(0, \mathfrak{R}(u^{k+1}))\leq &2\beta^2\|A\|^2\|y^{k+1}-y^k\|^2\\
    &+2\gamma^2\|x^{k+1}-\bar{x}\|^2 +\frac{1}{\beta^2}\|\lambda^{k+1}-\lambda^k\|^2.
\end{aligned}
\end{equation}
\end{lemma}
\begin{proof}
    By the optimality conditions of the subproblems of ADMM scheme \eqref{alg:classical ADMM1}, we have 
    \begin{align*}
\begin{bmatrix}
\beta A^{\top}(y^k-y^{k+1})+s^{k+1}\\
0\\ \frac{1}{\beta}(\lambda^{k+1}-\lambda^k)
\end{bmatrix}\in \mathfrak{R}(u^{k+1}),
    \end{align*}
    where \begin{align*}
        s^{k+1}=\nabla f(\bar{x}) + \nabla^2 f(\bar{x})(x^{k+1} - \bar{x})-\nabla f(x^{k+1}).
    \end{align*}
    By Proposition \ref{prop:local-convergence}, $\{x^{k+1}\}$ converges to $\bar{x}$, hence, for given $\gamma>0$, there exists $K\in\mathbb{N}$ such that for all $k\geq K$, $\|s^{k+1}\|\leq\gamma\|x^{k+1}-\bar{x}\|$.
     Then, by Cauchy-Schwarz inequality, we obtain \eqref{ineq-T}.
\end{proof}

Based on Lemma \ref{lem-T-solution} and Lemma \ref{lem-ineq-T}, we now present the
local linear convergence result for the ADMM scheme \eqref{alg:classical ADMM1} applied to Problem \eqref{prob:linear-eq}. For \(v:=(y,\lambda)\in\mathbb R^m\times\mathbb R^m\) and a nonempty closed set
\(\mathcal V\subseteq\mathbb R^m\times\mathbb R^m\), define
\[
    \|v\|_\beta^2:=\beta\|y\|^2+\frac{1}{\beta}\|\lambda\|^2,\quad \mathrm{dist}_\beta(v,\mathcal V)
    :=
    \inf_{\widehat v\in\mathcal V}
    \|v-\widehat v\|_\beta.
\]
\begin{proposition}\label{prop-linear}
    Assume that $(\bar{x}, \bar{\lambda})$ satisfies Assumption~\ref{assp:ADMM} and $\mathcal{C}$ is a polyhedral convex set. Let the sequence $(x^{k+1},y^{k+1},\lambda^{k+1})$  be generated by ADMM scheme \eqref{alg:classical ADMM1} under the conditions of Proposition \ref{prop:local-convergence}. Set
    \[
         v^k:=(y^k,\lambda^k),
        \qquad
        \mathcal V^*
        :=
        \{A\bar x\}\times\Lambda_\beta^*.
    \]
    Then, $\mathrm{dist}_{\beta}(v^k,\mathcal{V}^*)$ converges Q-linearly to 0. Consequently,
    \(\mathrm{dist}(v^k,\mathcal V^*)\) and
    \(\{x^k\}\) converge R-linearly.
\end{proposition}
\begin{proof}
Let $u^k:=(x^k,y^k,\lambda^k)$ and $\mathcal U^*:=\{(\bar x,A\bar x)\}\times\Lambda_\beta^*$. 
By Proposition~\ref{prop:local-convergence},
\(u^k\to u^\infty =(\bar x,A\bar x,\widehat\lambda)\in\mathcal U^*\),
where
\(\widehat\lambda\in\operatorname{int}\Lambda_\beta\).
Thus, Lemma~\ref{lem-T-solution} implies that
\(\mathfrak R^{-1}(0)\) coincides locally with
\(\mathcal U^*\) around \(u^\infty\). Since $\mathcal{C}$ is a polyhedral convex set and $g$ is polyhedral convex, 
$\mathfrak{R}$ is a piecewise linear multifunction (see the discussion preceding Proposition~\ref{prop-error}). Then, Proposition~\ref{prop-error} yields constants
\(K_1\in\mathbb N\) and \(\kappa_1>0\) such that
\begin{align}\label{prop-linear-ineq-1}
    \operatorname{dist}(u^{k+1},\mathcal U^*)
	\leq
	\kappa_1
	\operatorname{dist}
	\bigl(0,\mathfrak R(u^{k+1})\bigr),
	\qquad k\geq K_1.
\end{align}
The set \(\mathcal V^*\) is nonempty and closed; hence, for
each \(k\), there exists
    \[
        \bar v^k
        =
        (\bar y,\bar\lambda^k)
        \in
        \operatorname*{argmin}_{v\in\mathcal V^*}
        \|v^k-v\|_\beta,
        \qquad
        \bar y:=A\bar x.
    \]
Since
    \(\bar\lambda^k\in\Lambda_\beta^*\), inequality~\eqref{ineq-fund-2}, applied
    with \(\lambda^*=\bar\lambda^k\), gives
    \begin{equation}\label{ineq-v_beta}
         \begin{aligned}
        &\|v^k-\bar v^k\|_\beta^2
        -
        \|v^{k+1}-\bar v^k\|_\beta^2\\
        &\quad\geq
        \sigma\|x^{k+1}-\bar x\|^2
        +
        \beta\|y^{k+1}-y^k\|^2
        +
        \frac{\beta-2\beta_0}{2\beta^2}
        \|\lambda^{k+1}-\lambda^k\|^2.
    \end{aligned}
    \end{equation}
On the other hand, by Lemma~\ref{lem-ineq-T}, for any fixed \(\gamma>0\),
    there exists \(K_2\in\mathbb N\) such that, for all \(k\geq K_2\),
    \begin{equation}\label{ineq-lem11}
        \begin{aligned}
        \mathrm{dist}^2\bigl(0,\mathfrak R(u^{k+1})\bigr)
        \leq{}&
        2\beta^2\|A\|^2
        \|y^{k+1}-y^k\|^2\\
        &+
        2\gamma^2\|x^{k+1}-\bar x\|^2
        +
        \frac{1}{\beta^2}
        \|\lambda^{k+1}-\lambda^k\|^2.
    \end{aligned}
    \end{equation}
    Therefore, with
    \[
        \kappa_2
        :=
        \max
        \left\{
            2\beta\|A\|^2,\,
            \frac{2\gamma^2}{\sigma},\,
            \frac{2}{\beta-2\beta_0}
        \right\},
    \]
    inequalities~\eqref{ineq-v_beta} and~\eqref{ineq-lem11} yield
    \begin{equation}\label{ineq-dist-R}
         \begin{aligned}
        \mathrm{dist}^2\bigl(0,\mathfrak R(u^{k+1})\bigr)
        &\leq
        \kappa_2
        \left(
            \|v^k-\bar v^k\|_\beta^2
            -
            \|v^{k+1}-\bar v^k\|_\beta^2
        \right)\\
        &\leq
        \kappa_2
        \left(
            \mathrm{dist}_\beta^2(v^k,\mathcal V^*)
            -
            \mathrm{dist}_\beta^2(v^{k+1},\mathcal V^*)
        \right)
    \end{aligned}
    \end{equation}
    for all \(k\geq K_2\). The second inequality follows from
    \[
        \|v^k-\bar v^k\|_\beta
        =
        \mathrm{dist}_\beta(v^k,\mathcal V^*),\quad 
        \mathrm{dist}_\beta(v^{k+1},\mathcal V^*)
        \leq
        \|v^{k+1}-\bar v^k\|_\beta.
    \]
Set $M_\beta:=\max\{\beta,\beta^{-1}\}$. By the definition of \(\mathcal U^*\) and \(\mathcal V^*\),
    \[
    \begin{aligned}
        \mathrm{dist}_\beta^2(v^{k+1},\mathcal V^*)
        \leq
        M_\beta
        \mathrm{dist}^2(v^{k+1},\mathcal V^*)\leq
        M_\beta
        \mathrm{dist}^2(u^{k+1},\mathcal U^*).
    \end{aligned}
    \]
Set \(K_0:=\max\{K_1,K_2\}\).    Combining this with~\eqref{prop-linear-ineq-1} and~\eqref{ineq-dist-R}, for every \(k\geq K_0\), we obtain
    \[
    \begin{aligned}
        \mathrm{dist}_\beta^2(v^{k+1},\mathcal V^*)
        \leq{}&
        M_\beta\kappa_1^2\kappa_2
        \left(
            \mathrm{dist}_\beta^2(v^k,\mathcal V^*)
            -
            \mathrm{dist}_\beta^2(v^{k+1},\mathcal V^*)
        \right).
    \end{aligned}
    \]
    Let $c_\beta
        :=
        M_\beta\kappa_1^2\kappa_2.$
    It follows that
    \[
        \mathrm{dist}_\beta^2(v^{k+1},\mathcal V^*)
        \leq
        \frac{c_\beta}{1+c_\beta}
        \mathrm{dist}_\beta^2(v^k,\mathcal V^*).
    \]
    Since \(c_\beta/(1+c_\beta)\in(0,1)\),
    \(\mathrm{dist}_\beta(v^k,\mathcal V^*)\) converges Q-linearly.

    Next, set $m_\beta:=\min\{\beta,\beta^{-1}\}$.
    The norm equivalence
    \[
        m_\beta\|v\|^2
        \leq
        \|v\|_\beta^2
        \leq
        M_\beta\|v\|^2
    \]
    implies
    \[
        \mathrm{dist}(v^k,\mathcal V^*)
        \leq
        \frac{1}{\sqrt{m_\beta}}
        \mathrm{dist}_\beta(v^k,\mathcal V^*).
    \]
    Hence, \(\mathrm{dist}(v^k,\mathcal V^*)\) converges R-linearly.
    
    Finally, by~\eqref{prop-linear-ineq-1} and~\eqref{ineq-dist-R},
    \[
    \begin{aligned}
        \|x^{k+1}-\bar x\|
        \leq
        \mathrm{dist}(u^{k+1},\mathcal U^*)\leq
        \kappa_1
        \mathrm{dist}\bigl(0,\mathfrak R(u^{k+1})\bigr)\leq
        \kappa_1\sqrt{\kappa_2}\,
        \mathrm{dist}_\beta(v^k,\mathcal V^*).
    \end{aligned}
    \]
    Therefore, \(\{x^k\}\) converges R-linearly.
\end{proof}

\section{Examples}\label{sec:example}
In this section, we present three illustrative examples and an application-oriented verification for a class of possibly nonconvex quadratic programs to demonstrate the scope of the local convergence results for ADMM scheme \eqref{alg:classical ADMM1} applied to Problem~\eqref{prob:com}. The three illustrative examples illustrate the role of Assumption~\ref{assp:ADMM} and the scope of the proposed local convergence theory, while the application-oriented verification shows how the second-order condition in Assumption~\ref{assp:ADMM} can be checked for quadratic programs with polyhedral convex feasible sets.

\subsection{Illustrative examples}
In this subsection, we present three examples demonstrating the local convergence behavior of ADMM scheme \eqref{alg:classical ADMM1} for solving Problem~\eqref{prob:com}. The first example demonstrates the local linear convergence of \eqref{alg:classical ADMM1} under the condition specified in Assumption~\ref{assp:ADMM} with $\mathcal{C}$ being a polyhedral convex set. The second example shows that if Assumption~\ref{assp:ADMM} is not satisfied, one cannot guarantee the local convergence of ADMM scheme \eqref{alg:classical ADMM1}. The third example highlights that, even when Assumption~\ref{assp:ADMM}  holds, the global convergence of ADMM scheme \eqref{alg:classical ADMM1} cannot be guaranteed.

\begin{example}
    Consider the problem 
    \begin{align}\label{exp-prob1}
        \min_{x\in[-\frac{1}{4},\frac{1}{4}]}~-\frac{1}{2}x^2+\vert x\vert.
    \end{align} 
    This problem can be  equivalently reformulated as 
    Problem \eqref{prob:linear-eq} with \[f(x):=-\frac{1}{2}x^2,~g(y):=\vert y\vert,~A=[1],~\mathcal{C}:=[-\frac{1}{4},\frac{1}{4}].\] 
    The ADMM iterative scheme \eqref{alg:classical ADMM1} in this case can be formulated explicitly as 
    \begin{align*}
        x^{k+1}&=\min\{\frac{1}{4},\max\{-\frac{1}{4},\frac{\beta y^k-\lambda^k}{\beta-1}\}\},\\
y^{k+1}&=\operatorname{prox}_{\vert\cdot\vert/\beta}(x^{k+1}+\frac{\lambda^k}{\beta}),\\
        \lambda^{k+1}&=\lambda^k+\beta(x^{k+1}-y^{k+1}).
    \end{align*}
Note that $\bar{x}=0$ is the unique global minimizer of Problem \eqref{exp-prob1}, and the pair $(\bar{x},\bar{\lambda})$ satisfies the first-order optimal condition \eqref{first order condition1} with $\bar{\lambda}=0$. Moreover, $\nabla^2 f(\bar{x})=-1$ satisfies Assumption~\ref{assp:ADMM} with the subspace $V_{\mathcal{C}}\cap \mathcal{S}_A(\bar{x}\mid\bar{\lambda})=\{0\}$. Hence, the strong variational sufficiency holds for Problem \eqref{exp-prob1} with respect to $(\bar{x},\bar{\lambda})$. Specifically, as
\begin{align}\label{Moreau-l1}
        M^{1/\beta}_{\vert\cdot\vert}(x)=\begin{cases}
            \frac{\beta}{2}x^2,&\vert x\vert \leq \frac{1}{\beta}\\
            \vert x\vert -\frac{1}{2\beta}, &\vert x\vert >\frac{1}{\beta}
        \end{cases}~,
    \end{align}
then for any $\beta>2$, $l_{\beta}(x,\lambda)$ is $(\beta-1)$ strongly convex in $x$ relative to $[-\frac{1}{2\beta},\frac{1}{2\beta}]$ when $\lambda\in[-\frac{1}{2},\frac{1}{2}]$ and is concave in $\lambda$ when $x\in [-\frac{1}{2\beta},\frac{1}{2\beta}]$.

Figure \ref{Fig.1} displays the average results over 50 independent runs with random  initial points $(y^0,\lambda^0)$ satisfying \eqref{the:initial} in Proposition \ref{prop:local-convergence}
with $\Lambda_{\beta}=[-\frac{1}{2},\frac{1}{2}]$ for  different $\beta>2$. This figure clearly shows that the ADMM iteration \eqref{alg:classical ADMM1} converges linearly in this case, which aligns with the local linear convergence guarantee of Proposition \ref{prop-linear}. Notably, this convergence is not covered
by the existing ADMM convergence results, as the problem is
nonconvex and \(\mathcal C\neq\mathbb R^n\). 

\begin{figure}[htbp]
\centering
\includegraphics[width=1\textwidth]{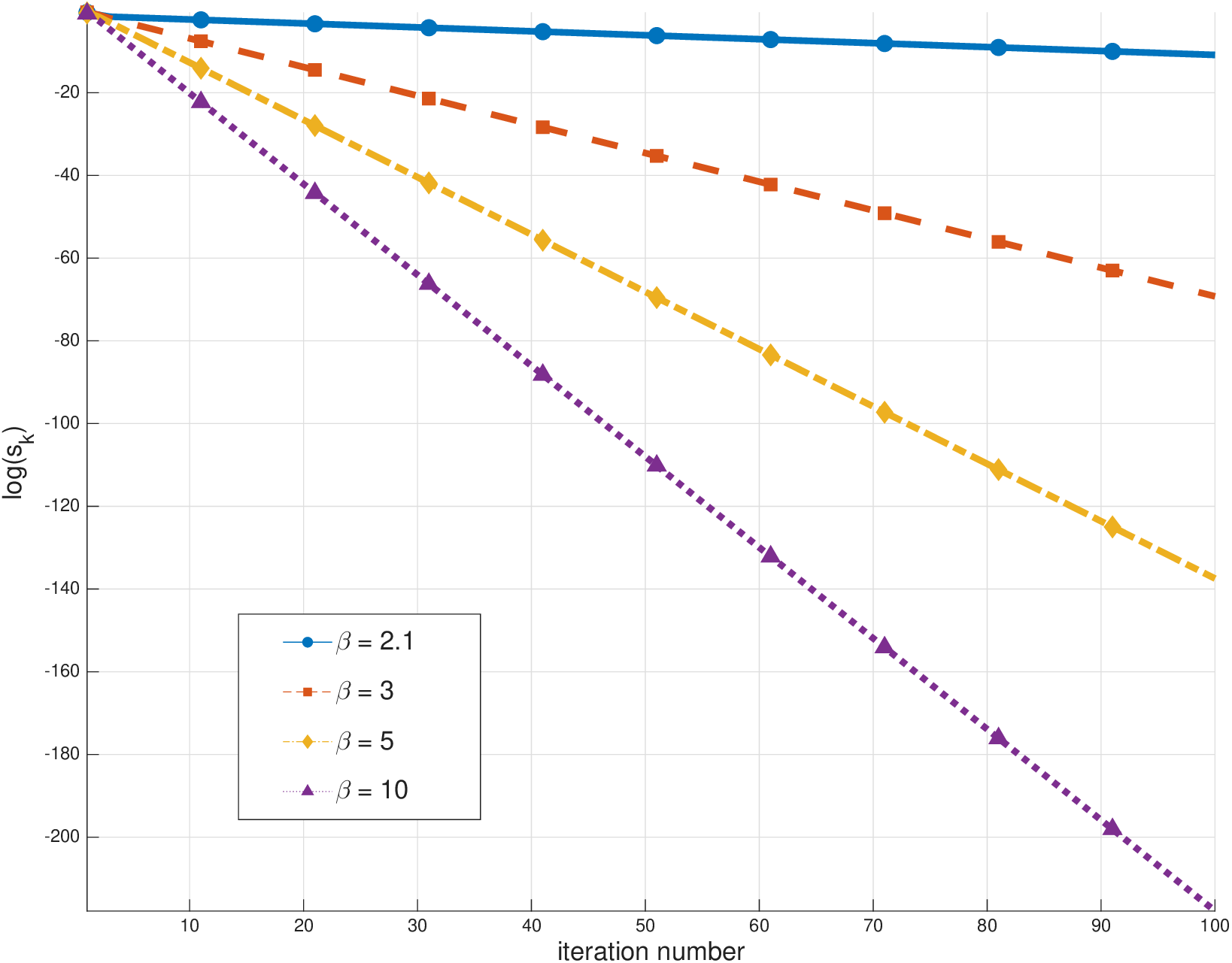}  
\captionsetup{justification=centering}   
\caption{ $\log(s_k)$ versus the iteration number, where $s_k:=\|x^k-\bar{x}\|+\|y^k-A\bar{x}\|+\|\lambda^k-\bar{\lambda}\|$.}
\label{Fig.1}
\end{figure}
\FloatBarrier
\end{example}

Next, we provide two illustrative examples to show the tightness of our local convergence analysis. 

\begin{example}
Consider the problem
\begin{align}\label{exp-prob2}
        \min_{x\in\mathbb{R}^2}~x_1\cos(x_2)+\vert x_1\vert.
    \end{align} 
    This problem can be  equivalently reformulated as 
    Problem \eqref{prob:linear-eq} with
 \[f(x_1,x_2):=x_1\cos(x_2),\quad g(y):=\vert y\vert, \quad A:=[1,0]\in\mathbb{R}^{1\times 2},\quad \mathcal{C}:=\mathbb{R}^2.\]    Note that the global minimizers of Problem \eqref{exp-prob2} form the set \[ \mathcal{M}=\mathcal{M}_1\cup\mathcal{M}_2^+\cup\mathcal{M}_2^-, \] where $\mathcal{M}_1=\{(0,t)\mid t\in\mathbb{R}\}$,  $\mathcal{M}_2^+=\{(x_1,(2k+1)\pi)\mid x_1>0,\ k\in\mathbb{Z}\}$,  and $\mathcal{M}_2^-=\{(x_1,2k\pi)\mid x_1<0,\ k\in\mathbb{Z}\}$. 
Let $\bar{x}:=[0,0]^{\top}$ and $\bar{\lambda}:=-1$. The pair $(\bar{x},\bar{\lambda})$ satisfies the first-order optimal condition \eqref{first order condition1}, but fails to satisfy Assumption~\ref{assp:ADMM}, since $\nabla^2 f(\bar{x})$ is not positive definite on $\mathcal{S}_A(\bar{x}\mid \bar{\lambda})=\mathbb{R}^2$. In this case, ADMM scheme~\eqref{alg:classical ADMM1} may generate a nonconvergent
sequence. Indeed, for any \(\beta>0\), let
\((y^0,\lambda^0)=(0,-1)\). By suitably choosing the minimizers
of the \(x\)-subproblems, the generated sequence satisfies
\begin{align*}
    (x^{2k+1},y^{2k+1},\lambda^{2k+1})
	&=
	\left(\left(\frac{2}{\beta},\pi\right),0,1\right),\\
	(x^{2k+2},y^{2k+2},\lambda^{2k+2})
	&=
	\left(\left(-\frac{2}{\beta},0\right),0,-1\right)
\end{align*}
for all \(k=0,1,\ldots\). Hence, without Assumption~\ref{assp:ADMM}(ii), ADMM may fail to converge
even when \((y^0,\lambda^0)=(A\bar x,\bar\lambda)\).
\end{example}

\begin{example}
Consider the problem \begin{align}\label{exp-prob3}
        \min_{x\in[-1,1]}~-\frac{1}{3}\vert x\vert ^3+\vert x\vert .
    \end{align} 
    This problem can be  equivalently reformulated as 
    Problem \eqref{prob:linear-eq} with 
    \[f(x):=-\frac{1}{3}\vert x\vert ^3, \quad g(y):=\vert y\vert ,\quad A:=[1],\quad \mathcal{C}:=[-1,1].\] 
Note that $\bar{x}=0$ is the unique global minimizer of Problem \eqref{exp-prob3}, and the pair $(\bar{x},\bar{\lambda})$ satisfies the first-order optimal condition \eqref{first order condition1} with $\bar{\lambda}=0$. Moreover, $\nabla^2 f(\bar{x})=0$ satisfies Assumption~\ref{assp:ADMM} with the subspace $V_{\mathcal{C}}\cap \mathcal{S}_A(\bar{x}\mid\bar{\lambda})=\{0\}$. Hence,  the strong variational sufficiency holds for Problem \eqref{exp-prob3} with respect to $(\bar{x},\bar{\lambda})$. Specifically, by \eqref{Moreau-l1}, for any $\beta>1$, $l_{\beta}(x,\lambda)$ is $(\beta-1/\beta)$ strongly convex in $x$ relative to $[-\frac{1}{2\beta},\frac{1}{2\beta}]$ when $\lambda\in[-\frac{1}{2},\frac{1}{2}]$ and is concave in $\lambda$ when $x\in [-\frac{1}{2\beta},\frac{1}{2\beta}]$. For a given $\beta>1$, if the initial point $(y^0,\lambda^0)$ does not satisfy \eqref{the:initial}, the sequence generated by \eqref{alg:classical ADMM1} for this problem may fail to converge. For example, let $\beta=2$ and take the initial point $(y^0,\lambda^0):=(0,-1)$. Then, the generated sequence forms a nonconvergent two-cycle:
\[
	(x^{2k+1},y^{2k+1},\lambda^{2k+1})=(1,0,1),
	\qquad
	(x^{2k+2},y^{2k+2},\lambda^{2k+2})=(-1,0,-1),
\]
for all \(k=0,1,\ldots\). We conclude that the global convergence of ADMM scheme \eqref{alg:classical ADMM1} cannot be guaranteed under Assumption~\ref{assp:ADMM}.
\end{example}

\subsection{Verification for a class of quadratic programs}

We next analyze the quadratic programming model~\eqref{QP} introduced in
Section~\ref{sec-intro}. Let
\[
	\mathcal{Q}
	:=
	\{x\in\mathbb R^n\mid Ex=e,\ Gx\leq h\}.
\]
As observed in the introduction, Problem~\eqref{QP} is a special case of
Problem~\eqref{prob:com} with
\[
	f(x):=\frac{1}{2}x^\top Hx+b^\top x,\qquad
	g(y):=\delta_{\mathcal{Q}}(y),\qquad
	A=I_n,\qquad
	\mathcal C=\mathbb R^n.
\]
Let \(\bar x\) be a local minimizer of Problem~\eqref{QP}, and let
\(\bar\lambda\in N_{\mathcal{Q}}(\bar x)\) satisfy
\[
	H\bar x+b+\bar\lambda=0.
\]
Since \(\nabla^2 f(x)=H\) for all \(x\in\mathbb R^n\), the Hessian
continuity requirement in Assumption~\ref{assp:ADMM}(ii) is automatically satisfied.
It remains to verify the positive-definiteness requirement on
\(\mathcal S_A(\bar x\mid\bar\lambda)\).
Let
\[
	I(\bar x):=\{i\mid G_i\bar x=h_i\}
\]
be the active index set. Since \(g=\delta_{\mathcal{Q}}\), it follows
from the definition of \(T_g\) that
\[
	T_g(\bar x\mid\bar\lambda)
	=
	T_{\delta_{\mathcal{Q}}}(\bar x\mid\bar\lambda)
	=
	\left\{
	w\in\mathbb R^n
	\ \mid\
	Ew=0,\;
	G_iw\leq0\ \text{for all }i\in I(\bar x),\;
	\langle\bar\lambda,w\rangle=0
	\right\}.
\]
Because \(A=I_n\) and \(\mathcal C=\mathbb R^n\), we have
\(V_{\mathcal C}=\mathbb R^n\) and
\[
	\begin{aligned}
	V_{\mathcal C}\cap
	\mathcal S_A(\bar x\mid\bar\lambda)
	=
	\mathcal S_A(\bar x\mid\bar\lambda)=
	T_{\delta_{\mathcal Q}}(\bar x\mid\bar\lambda)
	-
	T_{\delta_{\mathcal Q}}(\bar x\mid\bar\lambda)=
	\operatorname{span}
	T_{\delta_{\mathcal Q}}(\bar x\mid\bar\lambda),
	\end{aligned}
\]
where the last equality follows from the fact that
\(T_{\delta_{\mathcal Q}}(\bar x\mid\bar\lambda)\) is a convex cone. Therefore, the second-order requirement in Assumption~\eqref{assp:ADMM}(ii) reduces to
\[
	v^\top Hv>0
	\qquad
	\text{for every }
	0\neq v\in
	\operatorname{span}
	T_{\delta_{\mathcal Q}}(\bar x\mid\bar\lambda).
\]
Consequently, if \(H\) is positive definite on the linear span of
\(T_{\delta_{\mathcal Q}}(\bar x\mid\bar\lambda)\), then
Assumption~\ref{assp:ADMM} holds at \((\bar x,\bar\lambda)\). Since \(\mathcal C=\mathbb R^n\) is a polyhedral convex set, Proposition~\ref{prop-linear} yields the local linear convergence of the ADMM sequence, under the initialization and local trajectory conditions of
Proposition~\ref{prop:local-convergence}. This demonstrates
that the local theory developed in this paper applies to a practically
relevant class of possibly nonconvex quadratic programs.

\section{Conclusion}\label{sec-conclusion}
In this paper, we studied the local convergence of the standard ADMM scheme for a class of nonconvex composite optimization problems with convex constraints, where the objective consists of a smooth, possibly nonconvex, term and a polyhedral convex nonsmooth term composed with a linear mapping. Motivated by recent developments in variational analysis, we extended the notion of
strong variational sufficiency to this setting and provided an elementary, self-contained proof of a local strong convexity property of the Moreau envelope of polyhedral convex functions on the orthogonal complement of an appropriate subspace. This property, characterized by a local projection-based Hessian lower bound, allows us to verify the strong variational sufficiency of the reduced augmented Lagrangian under a suitable second-order
condition. It further leads to a descent inequality for the ADMM iterates that closely parallels the classical convex ADMM analysis. As a consequence, we established the local convergence of the ADMM scheme to a stationary primal-dual point for sufficiently large penalty parameters under suitable initialization and local trajectory conditions. When the constraint set is polyhedral convex, we further obtained Q-linear convergence of a weighted distance to the local primal-dual
solution set and R-linear convergence of the primal sequence.

Overall, this work clarifies the role of hidden convexity and variational structure in the local convergence of ADMM for nonconvex composite problems with constraints. It complements existing global convergence results based on KL-type arguments or error bounds, and provides a principled local theory aligned with recent advances in variational analysis. Possible directions for future research include extending the present framework to more general nonsmooth regularizers beyond the polyhedral case, deriving explicit contraction factors and sharper local rate estimates, and exploring the interaction between strong variational sufficiency and stochastic or inexact variants of ADMM arising in large-scale imaging and machine learning applications.

\backmatter


\bmhead{Authors' contributions}
X.-Y. Xie and Q. Li wrote the main manuscript text; all authors reviewed the manuscript.

\bmhead{Conflict of interest}
The authors declare no conflict of interest.




\end{document}